\documentclass[reqno]{llncs}

\usepackage{amsmath}
\usepackage{amsfonts}
\usepackage{amssymb}

\newcommand{\R}{\mathbb{R}}

\newcommand{\supp}{\text{supp}}
\newcommand{\spc}{\text{ }}

\newcommand{\x}{\textbf{x}}

\newcommand{\eoproof}{\hspace*{\fill} $\square$ \vspace{5pt}}

\usepackage{algorithm}
\usepackage[noend]{algpseudocode}
\usepackage{tikz}

\usepackage[margin=1in]{geometry}

\usepackage[colorlinks=true,breaklinks=true,bookmarks=false,urlcolor=blue,citecolor=blue,linkcolor=blue,bookmarksopen=false,draft=false]{hyperref}

\begin{document}



\title{Improved Linear Programs for Discrete Barycenters}

\author{Steffen Borgwardt\inst{1} \and Stephan Patterson\inst{2}}

\institute{\email{\href{mailto:steffen.borgwardt@ucdenver.edu}{steffen.borgwardt@ucdenver.edu}};
University of Colorado Denver 
\and \email{\href{mailto:stephan.patterson@ucdenver.edu}{stephan.patterson@ucdenver.edu}}; University of Colorado Denver
}

\date{}

\maketitle

\begin{abstract} Discrete barycenters are the optimal solutions to mass transport problems for a set of discrete measures. Such transport problems arise in many applications of operations research and statistics. The best known algorithms for exact barycenters are based on linear programming, but these programs scale exponentially in the number of measures, making them prohibitive for practical purposes.

In this paper, we improve on these algorithms. First, by using the optimality conditions to restrict the search space, we provide a reduced linear program that contains dramatically fewer variables compared to previous formulations. Second, we recall a proof from the literature, which lends itself to a linear program that has not been considered for computations. We show that this second formulation is the best model for data in general position. Third, we combine the two programs into a single hybrid model that retains the best properties of both formulations for partially structured data. 

We study these models through an analysis of their scaling in size, the hardness of the required preprocessing, and computational experiments. In doing so, we show that each of the improved linear programs becomes the best model for different types of data.
\end{abstract}


\section{Introduction}\label{sec:Intro}

Applications for optimization of mass transport for multiple marginals arise in a variety of fields: probabilistic Fr\'echet means in statistics \cite{mtbmmh-15,zp-17}, team matching in game theory \cite{ce-10,coo-15,p-14}, option prices and price equilibria in economics \cite{bhp-13,cmn-10}, and electron correlations in matter physics \cite{cfk-13}, to name a few. 
The variety of applications has led to considerable activity on the topic; a search for `optimal mass transport problems for several marginals' on GoogleScholar returns about 19,600 hits, including a seminal book by Villani \cite{v-09}, which has over 2,800 citations at the time of this writing. A search for `weighted Wasserstein barycenters' returns about 940 results.

A particularly noteworthy application is the design of deformable templates in manufacturing and image processing \cite{cd-14,jzd-98,ty-05}. An intuitive variant is in metal shaping: A number of sheets of metal have to be pressed (deformed) into multiple, different shapes (deformations). Each shape is modeled as a continuous probability distribution by comparing it to a flat sheet of metal. Physically, a {\em mean deformation} is a best shape for the initial sheet of metal, i.e., a shape that requires minimal energy to mold into all possible deformations. The mean deformation itself is represented as another continuous probability distribution, and the search for it can be formulated as an optimization problem to minimize the total energy cost required to obtain all desired deformations. Formal definitions of the deformations, their means, and this energy cost can be found in \cite{bll-11}. 

\subsection{Wasserstein Barycenters and the Discrete Barycenter Problem}

The {\em weighted Wasserstein barycenters} are optimal solutions to these problems. The distance between two probability measures $\mu$ and $\nu$ supported on $\R^d$ is calculated using the square of the quadratic Wasserstein distance
\[W_2 (\mu, \nu)^2 = \inf \Big \{ \int_{\R^d \times \R^d} \|x-y\|^2 d\gamma(x,y), \gamma \in \Pi(\mu,\nu) \Big \},\]
where $\Pi(\mu,\nu)$ denotes the set of all measures on $\R^d \times \R^d$ with $\mu, \nu$ as marginals. The set $\Pi(\mu,\nu)$ represents the set of transport plans between $\mu$ and $\nu$; see \cite{ac-11} and \cite{v-09} for more details. Then, given probability measures $P_1,{\dots},P_n$ on $\R^d$ and a strictly positive weight vector $\lambda\in \R^n_{+}$ with $\sum_{i=1}^n \lambda_i=1$, a Wasserstein barycenter is a probability measure $\bar P$ on $\R^d$ satisfying
\begin{equation}\label{eqn:weightedbary}
\varphi(\bar P):=\sum\limits_{i=1}^n \lambda_i W_2(\bar P, P_i)^2 = \inf\limits_{P\in \mathcal{P}^2(\R^d)} \sum\limits_{i=1}^n \lambda_i  W_2(P,P_i)^2,
\end{equation}
where $\mathcal{P}^2(\R^d)$ is the set of all probability measures on $\R^d$ with finite second moments. Informally, a barycenter $\bar P$ is a measure such that the total transport from $\bar P$ to all $P_i$ with respect to the quadratic Wasserstein distance is minimal.
 
Existence and uniqueness of barycenters for continuous probability measures is established in the study of the problem by Agueh and Carlier \cite{ac-11}, whose work is a foundation for much of the recent literature on barycenters. However, direct computation of barycenters for continuous probability measures has proven intractable outside of special cases \cite{cd-14,cp-16}, in part because an evaluation of the Wasserstein distance is computationally challenging in its own right. The computational effort of optimization on this distance creates a need for good approximate solutions. 

This is one reason for the interest in Wasserstein barycenters where the measures $P_i$ are supported on a {\em finite} set of points. We call such measures {\em discrete}. This setting arises through a discretization of the space of interest, as in \cite{coo-15}, which is one of the most powerful tools for heuristics to find an approximation of a barycenter. Additionally, many applications in operations research have naturally discrete support sets. In facility location problems, the support corresponds to transport destination locations; see the firehouse location example presented in \cite{abm-16}. The computation of a barycenter for measures with finite support is also used in distribution clustering \cite{ywwl-17}.

We formally state the problem of computing a  {\em discrete barycenter}, i.e., a barycenter for a given set of {\em discrete probability measures} $P_1,\dots,P_n$. We denote the finite support set of $P_i$  as $\supp(P_i) = \{\x_{ij}\big| j = 1,...,|P_i|\}$, where the size $|P_i| $ is the number of support points of measure $P_i$. 
Each $\x_{ij} \in \supp(P_i)$ has a corresponding mass $d_{ij}>0 $, and $\sum_{j=1}^{|P_i|} d_{ij} = 1$ for each $P_i$. In this paper, we are looking to improve exact algorithms to solve the:

\vspace{.1in}
\noindent{\bf Discrete Barycenter Problem}
 
 \vspace{.1in}
\textbf{Input:} $\,\,\,\,$ Discrete probability measures $P_1, \ldots, P_n$, weight vector $\lambda\in \R^n_{+}$  
\vspace{.1in}

\textbf{Output:} Discrete barycenter $\bar P$ for $P_1, \ldots, P_n$ and $\lambda$.\\

A proof of the existence of optimal solutions to the Discrete Barycenter Problem, as well as the linear programming formulation for producing these barycenters that is the starting point for our study, is established in \cite{abm-16}. 

\vspace{-.15in}
\subsection{The Set $S$ of Possible Barycenter Support Points}

Barycenters satisfy a property that is crucial for many applications: there exists an optimal transport to the measures that is {\em non-mass-splitting} \cite{ac-11,abm-16}, i.e., the mass of each barycenter support point is transported only to a single support point in each measure. 
This observation can be used to prove that a discrete barycenter is a discrete measure itself, supported on a subset of the set 
\begin{equation*}
S= \Big \{\sum\limits_{i=1}^n \lambda_i \x_{ij} \ : \  \x_{ij}\in \supp(P_i) \text{ for any }j = 1,\ldots,|P_i| \Big \}:=\{\x_1,\dots,\x_{|S|}\}.
\end{equation*}
This is the set of all convex combinations of support points, one from each measure $P_i$, given by the fixed $\lambda_i$ \cite{abm-16}. We call its elements $\x_k$ the {\em weighted means}. (A single-indexed $\x_k$ always refers to an element in $S$, in contrast to the double-indexed $\x_{ij}$ for support points of the original measures.) The size of $S$ is bounded by the product of the sizes of the input measures and thus may grow exponentially in the number of input measures. The linear program in \cite{abm-16} for the computation of discrete barycenters is based on finding the optimal mass on each `possible support point' in $S$. Thus, the exponential scaling of $S$ can be an extreme challenge for practical computations. 

The actual size of $S$ depends greatly on the underlying data. In this paper, we consider different types of data. We say the measures $P_1,\dots,P_n$ are {\em in general position} if different combinations of $\x_{ij}\in \supp(P_i)$ always induce different weighted means $\x_k$; this is the worst-case scenario producing an $S$ of largest possible size: $|S|=\prod_{i=1}^n |P_i|$. Additionally, we consider {\em structured data}: the support sets of the $P_i$ have known properties. Structured data has two beneficial computational properties: significant repetition in the weighted means creates a smaller set $S$, and the a priori knowledge allows for easier construction of the set. In particular, we consider a {\em highly structured} case in which all measures are supported on the same regular $d$-dimensional grid. 
Finally, the {\em partially structured} case refers to measures with support sets that can be partitioned into a part in general position and a part that is structured. Specifically, we study a situation in which the support of each measure is split into a part contained in a regular grid and a part outside of the grid. 

\vspace{-.15in}
\subsection{Linear Programming for Discrete Barycenters}\label{sec:sec13}

In contrast to the worst-case, exponential-sized possible support set, there always exists a barycenter with provably {\em sparse support}.
More precisely, there is a barycenter $\bar P$ with
\begin{equation}\label{sparseequ}
|\supp(\bar{P})| \leq (\sum_{i=1}^n |P_i|)  - n + 1.
\end{equation}

This is a tiny fraction of the number of possible support points in $S$, whose size is bounded by the product, rather than the sum, of the sizes of the original measures \cite{abm-16}. Such extreme sparsity is computationally promising, although no strategy to select the optimal support points from the large set $S$ is known. The sparsity can form the basis of an approximation scheme in two primary ways: selecting a relatively large, representative set of possible support points, or by choosing a small initial set and updating the possibilities after a fixed number of iterations \cite{ywwl-17}. Another approach is a strongly polynomial $2$-approximation algorithm based on the restriction of the set of support points for an approximate barycenter to the union of supports of the measures $P_1,\dots,P_n$ \cite{b-17}. In a different, heuristic style of approximation, algorithms based on projection and utilizing an entropic regularization term produce solutions with qualitative smoothing and full support \cite{bccnp-14,cd-14}. In this paper, we focus on exact barycenters, thus maintaining sparsely supported solutions. Our goal is to devise linear programming formulations with fewer variables and constraints. The formulations presented are also useful in LP-based approximation schemes. 

When representing a barycenter $\bar P$, we use values $z_k$, $ k=1,\ldots,|S|$, to denote the mass on support point $\x_k\in S$. Further, the values $y_{ijk}$ denote mass transported from $\x_k\in S$ to $\x_{ij} \in \supp(P_i)$ (for all $i=1,\dots,n$ and $j=1,\dots,|P_i|$).
With this notation, the transportation cost  (\ref{eqn:weightedbary}) in the discrete setting can be written as: 
\begin{equation}\label{eqn:LPobjective}
\varphi(\bar P):= \inf\limits_{P\in \mathcal{P}^2(\R^d)} \sum\limits_{i=1}^n \lambda_i  W_2(P,P_i)^2 = \min \sum_{i=1}^n\lambda_i\sum_{j = 1}^{|P_i|}\sum_{k=1}^{|S|}\|\x_k - \x_{ij} \|^2 y_{ijk}.
\end{equation}
Note that (\ref{eqn:LPobjective}) is a linear objective function, as the $\x_k$ and $\x_{ij}$ are part of the input. 

It is open whether the Discrete Barycenter Problem can be solved in polynomial time. The best known algorithms for exact barycenters are based on linear programming \cite{abm-16,coo-15,sld-18,ywwl-17}. The program is established as follows: Use variables $z_k$ to measure the (unknown) masses of a barycenter supported on the elements $\x_k\in S$; thus the variables $z_k$ themselves define the barycenter. The mass $z_k$ is transported to each measure $P_i$, the amount of which is indicated by the variables $y_{ijk}$. This yields capacity constraints $\sum_{j=1}^{|P_i|} y_{ijk}  =  z_k$ (for all  $i=1,\dots,n$ and $k=1,\ldots,|S|$), as the mass transport to each $P_i$ cannot exceed the available barycenter mass. Further, each support point $\x_{ij}$ in each measure $P_i$ receives exactly its mass $d_{ij}$ from the barycenter support points, which can be stated as $\sum_{k=1}^{|S|} y_{ijk}  = d_{ij}$ (for all $i=1,\dots,n$ and $j=1,\dots,|P_i|$). 

Thus a linear program for the computation of a barycenter may be formulated as: 
\begin{equation*}\label{baryLP}
\begin{array}{crll}
\tag{original} \mathrm{min}   &\multicolumn{3}{l}{ \sum\limits_{i=1}^n\lambda_i\sum\limits_{j = 1}^{|P_i|}\sum\limits_{k=1}^{|S|}\|\x_k - \x_{ij} \|^2 y_{ijk}} \nonumber \\
\mathrm{s.t.} &\sum\limits_{j=1}^{|P_i|} y_{ijk} & = \spc z_k,  &\forall i=1,\ldots,n,\spc\forall k=1,\ldots,|S|,\nonumber\\ 
& \sum\limits_{k=1}^{|S|} y_{ijk} & = \spc d_{ij}, &\forall i=1,\ldots,n,\spc\forall j=1,\ldots,|P_i|,\nonumber\\
& y_{ijk}  &\geq  \spc0, &\forall i=1,\ldots,n,\spc\forall j=1,\ldots,|P_i|,\spc\forall k=1,\ldots,|S|.\\ 
\end{array}
\end{equation*}

Since the $z_k$ variables represent a barycenter, which itself is a probability measure, we do also require $z_k \geq 0$ and $\sum_{k=1}^{|S|} z_k = 1$; however, these properties are naturally enforced by the capacity constraints and the fact that $\sum_{j=1}^{|P_i|} d_{ij} = 1$. Any optimal vertex of the LP has sparse support, i.e., it satisfies (\ref{sparseequ}). See \cite{abm-16,coo-15} for more details.  We summarize this information in the following proposition.

\begin{proposition}\label{thm:5eq}
The Discrete Barycenter Problem can be solved using LP (\ref{baryLP}). Any optimal vertex corresponds to a sparse barycenter.
\end{proposition}

One of the important observations about LP (\ref{baryLP}) is that its size may scale exponentially in the number $n$ of measures \cite{b-17}. For simplicity in denoting the worst-case scenario, assume $|P_i|=p_{\max}$ for all $i=1,\dots,n$. For measures $P_1,\dots,P_n$ with support points in general position, $|S|=\prod_{i=1}^n |P_i|=(p_{\max})^n$.  Then LP (\ref{baryLP}) has $n(p_{\max})^{n+1}+(p_{\max})^n$ variables and $n (p_{\max})^n+n p_{\max}$ equality constraints.

Most applications arise in dimension two or three, but the dimension does not actually appear as a factor in the exponential scaling of LP (\ref{baryLP}). Even for rather small input sizes, computations based on LP (\ref{baryLP}) already are challenging:  The computation of a barycenter for the firehouse location problem in \cite{abm-16}, i.e., for $n=8$ measures consisting of the same $p_{\max}=9$ support points (a structured data set), already took several minutes on a standard laptop (2.3 GHz Intel Core i7 MacBook Pro). 

\subsection{Contributions and Outline}

The poor scaling of LP (\ref{baryLP}) is the motivation for our research. We devise new, smaller LP formulations that allow exact computations for larger instances of the Discrete Barycenter Problem. In this paper, we improve on LP  (\ref{baryLP}) while retaining all optimal solutions. By using the optimality conditions of discrete barycenters in a better way, we present several ways to reduce the model size.

In {\bf Section \ref{sec:Form}}, we formally devise these improvements. First, we explain why only a strict subset of the variables $y_{ijk}$ is required for a better formulation LP (\ref{barymodLP}). We show that LP (\ref{barymodLP}) always is an improvement on LP (\ref{baryLP}): it has strictly fewer variables and the nonzero entries in the constraint matrix of LP (\ref{barymodLP}) are a strict subset of the nonzero entries in the corresponding columns of the constraint matrix of LP (\ref{baryLP}). 

Second, we turn to an alternative LP (\ref{LPw}) that finds a discrete barycenter. It was used in  \cite{abm-16,m-16} to show the existence of a sparse barycenter for all discrete barycenter problems. Due to an inherent, unavoidably exponential scaling that is independent of the underlying data, it has not been previously considered for computational purposes. We use this linear program to reveal that for two input measures, the discrete barycenter problem is a classical transportation problem of the type introduced by Ford and Fulkerson \cite{ff-56b}.

Third, we combine the two above LPs in a hybrid model, LP (\ref{LPhybrid}). The key idea is that the choice of model can be split into independent decisions for each combination of input support points. For partially structured data, a mix of the strategies of LP (\ref{barymodLP}) and LP (\ref{LPw}) through LP (\ref{LPhybrid}) retains the beneficial properties of both formulations.

{\bf Section \ref{sec:theory}} is dedicated to an analysis of the sizes of the new models. In Section \ref{sec:datageneral}, we discuss data in general position. We show that both LP (\ref{barymodLP}) and LP (\ref{LPw}) are dramatic improvements over LP (\ref{baryLP}), and that LP (\ref{LPw}) is the smallest model. Further, we show that a best implementation of LP (\ref{LPhybrid}) becomes identical to LP (\ref{LPw}). 

In Section \ref{sec:modgrid}, we discuss data supported in a $d$-dimensional regular grid. In this highly structured setting, the size of $S$ becomes polynomial in $n$ for fixed dimension $d$. This implies that LP (\ref{baryLP}) scales polynomially (and recall that LP (\ref{barymodLP}) and LP (\ref{LPhybrid}) provide further improvements), in strong contrast to the exponential scaling of LP (\ref{LPw}) for all data. We provide a formal proof that the Discrete Barycenter Problem (for evenly weighted measures) can be solved in strongly polynomial time in this setting.

{\bf Section \ref{sec:process}} is a study of the preprocessing required to set up the various LP formulations. In Section \ref{sec:hard}, we prove that the construction of all the LPs is hard (exponential unless $P=NP$), even if the resulting LPs are of polynomial size. More precisely, we show that a decision version of one small step of the necessary preprocessing is already NP-hard. This step may even be repeated exponentially many times for LPs (\ref{barymodLP}) and (\ref{LPhybrid}).

In Section \ref{sec:apriori}, we discuss how expert knowledge that data is supported on a regular grid can help. Under the additional assumption that the measures are evenly weighted, we devise a simple and efficient preprocessing routine to achieve a significant improvement in setup time for LP (\ref{baryLP}). This routine also avoids the inefficient preprocessing required for an exact setup of LPs (\ref{barymodLP}) and LP (\ref{LPw}), at the cost of a minor increase in linear program size.  

{\bf Section \ref{sec:comp}} contains our computational experiments. We use three representative types of data: a geospatial data set in general position, the well-known MNIST digits data set (c.f. \cite{lbh-98}), and a tailored data set that combines the properties of these two sets.

In Section \ref{sec:compgeneral}, we briefly exhibit the computational advantages of LP (\ref{LPw}) over LP (\ref{baryLP}) for data in general position. We use crime data for Denver County, publicly available as part of the Denver Open Data Catalog (\url{www.denvergov.org/opendata}), to construct such a data set. Specifically, we use the geographical locations of incidents in different months and years to represent a set of crime patterns. A barycenter for these crime patterns is readily interpreted as a set of locations for which police presence could lead to a fast crime response.

Section \ref{sec:digitcomp} shows our computational experiments for grid-structured data. We use the MNIST data set of handwritten digits (\url{http://yann.lecun.com/exdb/mnist/}), which is widely used for benchmarking in machine learning; see \cite{lbh-98} for more information. Each digit is supported on a subset of a $16 \times 16$ grid. A set of barycenters, one computed for the measures representing each digit $0,1,\dots,9$, allows the classification of a new measure as one of the digits -- it suffices to evaluate the cost of transport from each barycenter to the new measure and pick the smallest value. 

In our experiments, we consider various settings: Computations with and without expert knowledge (that the data is supported on a grid), and the impact of preprocessing on model sizes and computation times in these settings. The fastest results are achieved when using the preprocessing routines from Section \ref{sec:apriori} - a combination of fast preprocessing and LP sizes that are just slightly larger than exact formulations of LP (\ref{barymodLP}) or LP (\ref{LPhybrid}) leads to a best practical performance.

In Section \ref{sec:besthybrid}, we highlight the advantages of LP (\ref{LPhybrid}). We construct a data set similar to the MNIST digits data: The letter `i' is recorded in a combination of two grids of different coarseness -- a finer grid is used to record the position and detail of the i-dot, since most of the differences between images lie in the dot. Here, LP (\ref{LPhybrid}) greatly outperforms the others because of the ability to adapt the representation of $S$ to the different parts of the data. Treating the i-dot as data in general position and the rest of the letter as grid-structured gives the smallest model and fastest computation times. 

We finish with some concluding remarks and open questions in Section \ref{sec:conc}.

\section{Improved Linear Programs}\label{sec:Form}

In the following, we describe three ways to improve on the formulation of LP (\ref{baryLP}). The first one is a strict improvement, the second one is the smallest model for data in general position, and the third one is a hybrid of the former two.

\subsection{LP (\ref{barymodLP}): Optimality Conditions for $y$-Variables}\label{sec:ystart}

For an initial reduction in size, we note that variables can be dropped from LP (\ref{baryLP}) while keeping all optimal solutions in the feasible set. Due to the non-mass-splitting property, we know that any $\x_k = \sum_{i = 1}^n \lambda_i \x_{ij} \in S$ is the optimal support point for mass which is to be transported to all of the $\x_{ij}$ from which it is constructed. This implies that, in an optimal solution, $\x_k$ never transports to any $\x_{ij}$ not in a weighted mean calculation producing $\x_k$. Equivalently, in an optimal solution, $y_{ijk} = 0$ for all such pairs $(\x_k, \x_{ij})$, and we can therefore eliminate those $y_{ijk}$ from the formulation. 

Additionally, we note that the capacity constraints in LP (\ref{baryLP}), \[\sum\limits_{j=1}^{|P_i|} y_{ijk} = \spc z_k,  \forall i=1,\ldots,n,\spc\forall k=1,\ldots,|S|,\nonumber\\ \] could be rewritten in a way that eliminates the variables $z_k$. Specifically, for each support point $\x_k$, the transport to each $P_i$ has to be equal:
\[ \sum\limits_{j=1}^{|P_1|} y_{1jk}=  \sum\limits_{j=1}^{|P_2|} y_{2jk}=\dots=\sum\limits_{j=1}^{|P_n|} y_{njk}.\]

However, we do not implement this change. It would remove $|S|$ variables and constraints, only a minor benefit to the program size. Meanwhile, the dramatically increased number of nonzero entries in the constraint matrix has an overall negative effect on computations.


Instead, we develop a reduced formulation based solely on the removal of the extraneous $y$-variables. We first introduce some notation. Let $S_{ij}$ be the set of indices $k$ for which $\x_k \in S$ can be computed as a weighted mean of a set of support points that includes $\x_{ij}$.  Formally,
$$ S_{ij}=\{\; k \; : \;  \x_k=\lambda_i\x_{ij} + \sum\limits_{l=1, l\neq i}^n \lambda_l \x_{lj'} \text{ for some } \x_{lj'}\in \supp(P_l)\}.$$

Conversely, let $S_k$ be the set of index tuples $(i,j)$ of support points $\x_{ij}$ which contribute to a computation of $\x_k$, i.e., 
$$ S_k=\{\; (i,j) \; : \;  \x_k=\lambda_i\x_{ij} + \sum\limits_{l=1, l\neq i}^n \lambda_l \x_{lj'} \text{ for some } \x_{lj'}\in \supp(P_l)\}.$$

With this notation, LP (\ref{baryLP}) can be improved to a smaller formulation as follows.
\begin{equation*}\label{barymodLP}
\begin{array}{crll}
\tag{reduced} \mathrm{min} & \multicolumn{3}{l}{\sum\limits_{i=1}^n\lambda_i\sum\limits_{j = 1}^{|P_i|}\sum\limits_{k:k \in S_{ij}}\|\x_k - \x_{ij} \|^2 y_{ijk}} \nonumber \\
\mathrm{s.t.} &\sum\limits_{j:(i,j) \in S_k} y_{ijk}  &= \spc\spc z_k,  &\forall i=1,\ldots,n,\spc\forall k=1,\ldots,|S|\nonumber\\
&\sum\limits_{k:k \in S_{ij}} y_{ijk}  &= \spc\spc d_{ij}, &\forall i=1,\ldots,n,\spc\forall j=1,\ldots,|P_i|\\
&y_{ijk} & \geq  \spc\spc0, &\forall i=1,\ldots,n,\spc\forall j=1,\ldots,|P_i|,\spc \forall k \in S_{ij} \nonumber \\
\end{array}
\end{equation*}

Since the optimal vertices of LPs (\ref{baryLP}) and (\ref{barymodLP}) are in one-to-one correspondence, we can solve either of them to the same result. We obtain the following theorem. 

\begin{theorem}\label{thm:6eq}
The Discrete Barycenter Problem can be solved using LP (\ref{barymodLP}). Any optimal vertex corresponds to a sparse barycenter.
\end{theorem}

Note that at least one pair $(\x_k,\x_{ij})$, such that $\x_{ij}$ is not part of any weighted means construction of $\x_k$, exists for any dimension $\mathbb{R}^d$, as long as $|P_i| \geq 2$ for at least one $i = 1, \dots, n$. In other words, for non-trivial input there is an $(i,j)\notin S_k$ for some $k$. Thus LP (\ref{barymodLP}) always provides a strict reduction in the number of variables and the number of non-zero entries in the constraint matrix.

\begin{lemma}\label{lem:lessy}
For a set of measures $P_1,\dots,P_n$, at least one with $|P_i| \geq 2$, LP (\ref{barymodLP}) has strictly fewer variables than LP (\ref{baryLP}). Further, the nonzero entries of the constraint matrix for LP (\ref{barymodLP}) are a strict subset of the nonzero entries in the corresponding columns of the constraint matrix of LP (\ref{baryLP}).
\end{lemma}

In Section \ref{sec:theory}, we turn to the dramatic reduction in size from LP  (\ref{baryLP}) to LP  (\ref{barymodLP}) -- often several orders of magnitude -- in more detail.

\subsection{LP (\ref{LPw}): Fixed Transport}
An alternative linear program for the Discrete Barycenter Problem was used in \cite{abm-16,m-16} to prove the existence of a sparse barycenter, but it was not considered for computational purposes. However, as we will see in Sections \ref{sec:datageneral} and \ref{sec:compgeneral}, this is the best approach for data in general position.

The key idea is to treat each combination of original support points that gives a weighted mean separately, even if some combinations produce the same weighted mean $\x_k$. Applying this idea to LP (\ref {barymodLP}), a variable $z_k$ is used for each {\em combination} of original support points. Then $|S_k|=n$, as $S_k$ contains precisely one pair $(i,j)$ for each $i\leq n$. When mass is associated with variable $z_k$, it is also fully associated to all of the corresponding $y_{ijk}$ with $(i,j)\in S_k$. This implies that the variables $y_{ijk}$ can be eliminated through $y_{ijk}=z_k$ for all $(i,j)\in S_k$. Informally, assigning mass to $z_k$ gives rise to a {\em fixed transport} to the measures.

We now develop an LP formulation, in our own notation, based on this idea. First, we define the set of all combinations of original support points, one from each measure, as 
$$S^* = \{ (\x_{1j}, \ldots, \x_{nj}) \ : \ \x_{ij} \in \supp(P_i) \text{ for any }j = 1,\ldots,|P_i|\}:=\{s_1^*,\dots,s^*_{|S^*|}\}.$$
There is an intimate relation of $S^*$ and $S$: Each tuple $s^*_h=(\x_{1j}, \ldots, \x_{nj})$, $h=1,\dots,|S^*|$, corresponds to a set of original support points $\x_{ij}$, $i=1,\dots,n$, and their weighted mean $\x_k = \sum_{i=1}^n \lambda_i \x_{ij} \in S$. Generally, it is possible that multiple $s^*_h\in S^*$ are associated with the same $\x_k\in S$. In turn, each $\x_k\in S$ is associated with at least one $s^*_h \in S^*$. We immediately obtain $|S^*| \geq |S|$. For data in general position, $|S^*| = |S|$ and there is a bijection between the tuples  $s^*_h$ and weighted means $\x_k$.

Next, we introduce a variable $w_h$ for each $s^*_h$, $h = 1, \ldots, |S^*|$, representing mass associated with that combination. The corresponding cost $c_h$ of transporting a unit of mass is:
\begin{equation*}
c_h = \sum_{i=1}^n \lambda_i ||\x_k-\x_{ij}||^2.
\end{equation*}
Finally, we define sets $S_{ij}^*$ similarly to the sets $S_{ij}$.  $S_{ij}^*$ is the set of indices $h$ in $1, \ldots, S^*$ such that the $i^{th}$ component of $s_h^*$ is $\x_{ij}$. Formally, $$ S_{ij}^*=\{\; h \; : \;  s_h^*=(\dots,\x_{ij},\dots)\}.$$

We now have all the ingredients for the LP. Due to the fixed transport, the barycenter capacity constraints in LPs (\ref{baryLP}) and (\ref{barymodLP}) are naturally enforced and can be eliminated. All variables $y_{ijk}$ are also eliminated, and so we obtain the LP:
\begin{equation*}\label{LPw}
\begin{array}{crll}
\tag{general}  \mathrm{min}  &\multicolumn{3}{l}{\sum\limits_{h = 1}^{|S^*|} c_h w_h }\nonumber \\
\mathrm{s.t.  } &\sum\limits_{h:h \in S_{ij}^*} w_h &= \spc\spc d_{ij}, &\forall i=1,\ldots,n,\spc\forall j=1,\ldots,|P_i|\\
 &w_h  &\geq  \spc\spc 0, &\forall  h=1,\ldots,|S^*|.\nonumber
\end{array}
\end{equation*}

Barycenter masses $z_k$ and an optimal transport $y_{ijk}$ can readily be reconstructed from the $w_h$, and the optimal vertices of LPs (\ref{baryLP}) and (\ref{LPw}) again are in one-to-one correspondence. 

\begin{theorem}\label{thm:7eq}
The Discrete Barycenter Problem can be solved using LP (\ref{LPw}). Any optimal vertex corresponds to a sparse barycenter.
\end{theorem}

For a set of two discrete measures, the computation of a discrete barycenter can be modeled as a classical transportation problem; this is explained in \cite{m-16} through a transformation of LP (\ref{baryLP}). We demonstrate an easier way to prove this claim by reindexing LP (\ref{LPw}) as follows:

Denote the support elements as $\x_k \in \supp{(P_1)}$ with mass $d_k$ and $\x_l \in \supp{(P_2)}$ with mass $d_l$. Further, represent the weights as $\lambda_1=\lambda$ and $\lambda_2=1-\lambda$. For a given $\x_k$ and $\x_l$, an optimal barycenter support point is $\x=\lambda \x_k + (1-\lambda) \x_l$ and the corresponding cost of transport is \begin{align*} c_{kl}&=\lambda \|\x -\x_k\|^2 + (1-\lambda) \|\x -\x_l\|^2 = (\lambda (1-\lambda)^2+ \lambda^2(1-\lambda))\|\x_k -\x_l\|^2  =\lambda (1-\lambda)\| \x_k -\x_l\|^2. \end{align*}
With $w_{kl}$ denoting the mass associated to the combination of $\x_k$ and $\x_l$, the Discrete Barycenter Problem for $n=2$ can be written as the following transportation problem.
\begin{equation*}\label{LPtransport}
\begin{array}{crll}
\tag{transportation} \mathrm{min}  &\multicolumn{2}{l}{\sum\limits_{k=1}^{|P_1|}\sum\limits_{l=1}^{|P_2|} c_{kl} w_{kl}} \nonumber \\
\mathrm{s.t.  } &\sum\limits_{l=1}^{|P_2|} w_{kl}=& \spc d_{k}, &\forall k=1,\ldots,|P_1	|\\ 
&\sum\limits_{k=1}^{|P_1|} w_{kl} =& \spc d_{l}, &\forall l=1,\ldots,|P_2| \nonumber\\
&w_{kl}  \geq& \spc 0, &\forall k=1,\ldots,|P_1|, \spc\forall l=1,\ldots,|P_2|\nonumber
\end{array}
\end{equation*}

It is well-known that transportation problems can be solved in strongly polynomial time: Transportation problems have totally-unimodular constraint matrices and are a special case of minimum-cost flow problems \cite{o-93,t-86}. There are various specialized algorithms for transportation problems \cite{b-08,ks-95,tn-95}; see \cite{ms-16} for a survey. The best running time for our setting is due to \cite{ks-95}, with corresponding bound stated in Theorem \ref{thm:transport} in our notation.
\begin{theorem}\label{thm:transport}
Consider the Discrete Barycenter Problem for $n=2$ with $p_1=|P_1|$,  $p_2=|P_2|$, and w.l.o.g. $p_1\geq p_2$. Then the problem can be solved in strongly polynomial time $$O(p_1 \log p_1 \; (p_1p_2+p_2\log p_2)).$$
\end{theorem}

\subsection{LP (\ref{LPhybrid}): a Hybrid Approach to Variable Introduction}\label{sec:hyb}
The strategies of variable introduction presented in LP (\ref{barymodLP}) and LP (\ref{LPw}) are not mutually exclusive. For each tuple $s_h^*$ individually, we can decide whether to use a representation with $y,z$-variables as in LP (\ref{barymodLP}) or a representation with $w$-variables as in LP (\ref{LPw}). To this end, we partition the set $\{1,\dots, |S^*|\}$ into two index sets $(S^*)^y$ and $(S^*)^w$ to indicate for which  $s_h^*$ we use which representation.

The original support points $\x_{ij}$ can then receive mass in two ways: through a transport denoted by some $y_{ijk}$ and through a fixed transport of a mass $w_h$. First, we define $(S^*)^w_{ij}$ to be the set of indices $h$ such that the $i^{th}$ component of $s_h^*$ is $\x_{ij}$. Then $(S^*)^w_{ij}\subset (S^*)^w$ corresponds precisely to those combinations $s_h^*$ which imply a fixed transport to $\x_{ij}$. For a formal definition, we only have to restrict $S_{ij}^*$ to indices $h \in (S^*)^w$, i.e.,
$$ (S^*)^w_{ij}=\{\; h \; : \;  s_h^*=(\dots,\x_{ij},\dots), h \in (S^*)^w\}.$$
For a proper indexing of the transport corresponding to $(S^*)^y$, we need index sets that mirror $S_{ij}$ and $S_k$ as defined in Section \ref{sec:ystart}, but restricted to $(S^*)^y$. $S^y_{ij}$ contains the indices $k$ of all $\x_k \in S$ produced by the weighted means of tuples in $(S^*)^y$ that contain $\x_{ij}$. Formally,
$$ S_{ij}^y=\{\; k \in S\; : \;  \x_k=\lambda_i\x_{ij} + \sum\limits_{l=1, l\neq i}^n \lambda_l \x_{lj'} \text{ for } s^*_h=(\x_{1j'},\dots, \x_{ij},\dots,\x_{nj'}) \in (S^*)^y \}.$$
Note that the index pair $(i,j)$ is fixed in the definition of $S_{ij}^y$. (The use of $j'$ in the other $\x_{lj'}$ indicates that the corresponding $j'$ do not have to match $j$.) Further, $S^y_k$ contains those index pairs $(i,j)$ for which $\x_k \in S$ is the weighted mean of a tuple in $(S^*)^y$ that contains $\x_{ij}$. This gives
$$ S^y_k=\{\; (i,j) \; : \;  \x_k=\lambda_i\x_{ij} + \sum\limits_{l=1, l\neq i}^n \lambda_l \x_{lj'} \text{ for some } s^*_h=(\x_{1j'},\dots,\x_{ij},\dots,\x_{nj'}) \in (S^*)^y\}.$$

Now, we are ready to state an LP that allows for the split of $\{1,\dots, |S^*|\}$ into index sets $(S^*)^y$ and $(S^*)^w$:
\begin{equation*}\label{LPhybrid}
\begin{array}{crll}
\tag{hybrid} \mathrm{min}  &\multicolumn{3}{l}{ \sum\limits_{h:h \in (S^*)^w}c_h w_h+\sum\limits_{i=1}^n\lambda_i\sum\limits_{j = 1}^{|P_i|}\sum\limits_{k:k \in S^y_{ij}}\|\x_k - \x_{ij} \|^2 y_{ijk}}  \\
\mathrm{s.t.} &\sum\limits_{j:(i,j) \in S^y_k} y_{ijk}  &= \spc\spc z_k,  &\forall i=1,\ldots,n,\spc\forall k \in S^y_{ij}\\
&\sum\limits_{h:h \in (S^*)^w_{ij}} w_h + \sum\limits_{k:k \in S^y_{ij}} y_{ijk}  &= \spc\spc d_{ij}, &\forall i=1,\ldots,n,\spc\forall j=1,\ldots,|P_i|\\
&w_h  &\geq \spc\spc 0, &\forall  h \in (S^*)^w \\
&y_{ijk}  &\geq \spc\spc0, &\forall i=1,\ldots,n,\spc\forall j=1,\ldots,|P_i|,\spc \forall k \in S^y_{ij}.  \\
\end{array}
\end{equation*}

Correctness of this model is a direct consequence of Theorems \ref{thm:6eq} and \ref{thm:7eq}.

\begin{theorem}\label{thm:8eq}
The Discrete Barycenter Problem can be solved using LP (\ref{LPhybrid}). Any optimal vertex corresponds to a sparse barycenter.
\end{theorem}

In Section \ref{sec:besthybrid}, we show computations for an example where LP (\ref{LPhybrid}) is much smaller than the other `pure' LP formulations.

\section{Model Sizes}\label{sec:theory}

Next, we study the advantages of LPs (\ref{barymodLP}), (\ref{LPw}), and (\ref{LPhybrid}) over LP  (\ref{baryLP}) in model size, i.e., the number of variables and constraints, in two representative settings. We begin with data in general position and then turn to highly structured data, in which support is on a regular $d$-dimensional grid. 
\subsection{Data in General Position}\label{sec:datageneral}

For a simple notation, we assume that all $P_i$ have the same number of support points $|P_i|= p_{\max}$. Recall that data in general position occurs when all $(p_{\max})^n$ combinations of support points in the original measures produce $(p_{\max})^n$ different weighted means $\x_k$. Thus $|S|=|S^*|=(p_{\max})^n$.

LP (\ref{baryLP}) is set up with a variable $z_k$ for each of support point in $\x_k\in S$, as well as variables $y_{ijk}$ that indicate transport from $\x_k$ to each $\x_{ij}$. The result is an LP with $(p_{\max})^n+n(p_{\max})^{n+1}$ variables, $(p_{\max})^n$ of type $z_k$ and $n(p_{\max})^{n+1}$ of type $y_{ijk}$, and $n (p_{\max})^n+n p_{\max}$ constraints; recall the discussion at the end of Section \ref{sec:sec13}. All of the other formulations improve significantly on these numbers; the results are summarized in Table \ref{table:pmax}.

First, we turn to LP (\ref{barymodLP}). Because of the general position of the data, each $\x_k$ may transport only to exactly those $n$ points $\x_{ij}$ from which it is constructed. Thus, we reduce the number of variables $y_{ijk}$ for each $\x_k$ from $np_{\max}$ to $n$. The number of variables $z_k$ and the number of constraints remain unaffected. This gives a total of $(1+n)(p_{\max})^n$ variables. A representation of the reduction of the number of variables as a percentage highlights the improvement over LP (\ref{baryLP}). It is a function dominated by $p_{\max}$:
\[ 1-\dfrac{(1+n)(p_{\max})^n}{(p_{\max})^n+n(p_{\max})^{n+1}}=1-\dfrac{1+n}{1+n p_{\max}}= \dfrac{n(p_{\max}-1)}{1+n p_{\max}} \approx \frac{(p_{\max}-1)}{ p_{\max}}. \]

For example, for $ p_{\max}=256$ the fraction indicates about a $99.61\%$ reduction. 


Moving on to LP (\ref{LPw}), LP (\ref{LPw}) is a strict improvement over LP (\ref{barymodLP}). Informally, all of the $y$-variables from LP (\ref{baryLP}) and LP (\ref{barymodLP}) are eliminated. There is just a single variable $w_h$ for each $\x_k \in S$ with the index $h$ indicating the corresponding tuple in $S^*$. This further lowers the total number of variables  from $(1+n)(p_{\max})^n$ in LP (\ref{barymodLP}) to $(p_{\max})^n$ in LP (\ref{LPw}). 

A similar analysis shows that the reduction in the number of variables from LP (\ref{baryLP}) to LP (\ref{LPw}) is about $99.9\%$ for $n=4$ and $p_{\max} = 256$. The percentage reduction between the number of variables in LP (\ref{barymodLP}) and LP (\ref{LPw}) depends only on $n$ and gets arbitrary close to $100\%$ for increasing $n$; for $n=4$, it is already $80\%$.

\begin{table}[b]
\begin{center}
\begin{tabular}{|c|c|c|}\hline
LP formulation & Variables & Constraints \\ \hline
(\ref{baryLP}) & $n(p_{\max})^{n+1}+(p_{\max})^n$ & $n (p_{\max})^n+n p_{\max}$ \\ \hline
(\ref{barymodLP}) & $(1+n)(p_{\max})^n$ & $n (p_{\max})^n+n p_{\max}$ \\ \hline
(\ref{LPw}) & $(p_{\max})^n$ & $n p_{\max}$ \\ \hline
(\ref{LPhybrid}) & $(p_{\max})^n$ \text{ to } $(1+n)(p_{\max})^n$  & $n p_{\max}$ \text{ to } $n (p_{\max})^n+n p_{\max}$  \\ \hline
\end{tabular}
\end{center}
\caption{The number of variables and constraints of the different LP models for data in general position.} \label{table:pmax}
\end{table}

It is worth noting that, unlike LP (\ref{barymodLP}), LP (\ref{LPw}) also reduces the number of constraints from  $n(p_{\max})^{n}+np_{\max}$ in LPs  (\ref{baryLP}) and (\ref{barymodLP}) to  $n p_{\max}$: from exponential in $n$ to linear in $n$. This can be significant for the solution time, especially in cases in which the number of required variables is comparable between different formulations. For data in general position, LP (\ref{LPw}) always is the best model, by a significant margin. 

Finally, consider LP (\ref{LPhybrid}). For each of the $(p_{\max})^n$ different combinations of original support points, which here correspond to $(p_{\max})^n$ different weighted means   $\x_k$, we decide whether to introduce $w$-variables as in LP (\ref{LPw}) or $y, z$-variables as in LP (\ref{barymodLP}). Thus, the number of variables and constraints ranges between the values for  LP (\ref{LPw}) and LP (\ref{barymodLP}). As seen in the above, a best decision is to always introduce $w$-variables, which produces exactly LP (\ref{LPw}). 

\begin{theorem}\label{thm:hybisgen}
For data in general position, choosing $(S^*)^w=S^*$ and $(S^*)^y=\emptyset$ gives a minimal number of variables and constraints in LP (\ref{LPhybrid}). Then LP (\ref{LPhybrid}) and LP (\ref{LPw}) are identical.
\end{theorem}

\subsection{Grid-structured Data}\label{sec:modgrid}

LPs (\ref{baryLP}), (\ref{barymodLP}), and (\ref{LPhybrid}) have a significant advantage over LP (\ref{LPw}) in many practical applications: They are able to take advantage of repetition in the $\x_k$ constructed from different combinations $s_h^*$ of original support points. Recall that LPs (\ref{baryLP}) and  (\ref{barymodLP}) have variables $y$ and $z$ both indexed on $k = 1, \ldots, |S|$, where $|S|$ is the number of {\em distinct} weighted means. By contrast, LP (\ref{LPw}) introduces variables $w$ indexed on $h = 1, \ldots, |S^*|$. When there is heavy repetition of the $\x_k$, $|S^*|$ is much larger than $|S|$.

Here, we discuss the model sizes for (evenly weighted) measures supported on a $d$-dimensional regular grid $G_{\text{org}}$. This is the most common setting in optimal mass transport problems. The MNIST data set of handwritten digits is a prime example for such data; see \cite{lbh-98} for more information. The handwritten digits are recorded as a measure supported on a subset of a $16 \times 16$ grid. Scaling the masses to sum to one makes this a probability measure; see Figure \ref{fig:example8} (left) for an example. The shades of grey indicate the amount of mass at the support points; darker points hold more mass.

\begin{figure}[t]
\begin{center}
\includegraphics[scale = .35]{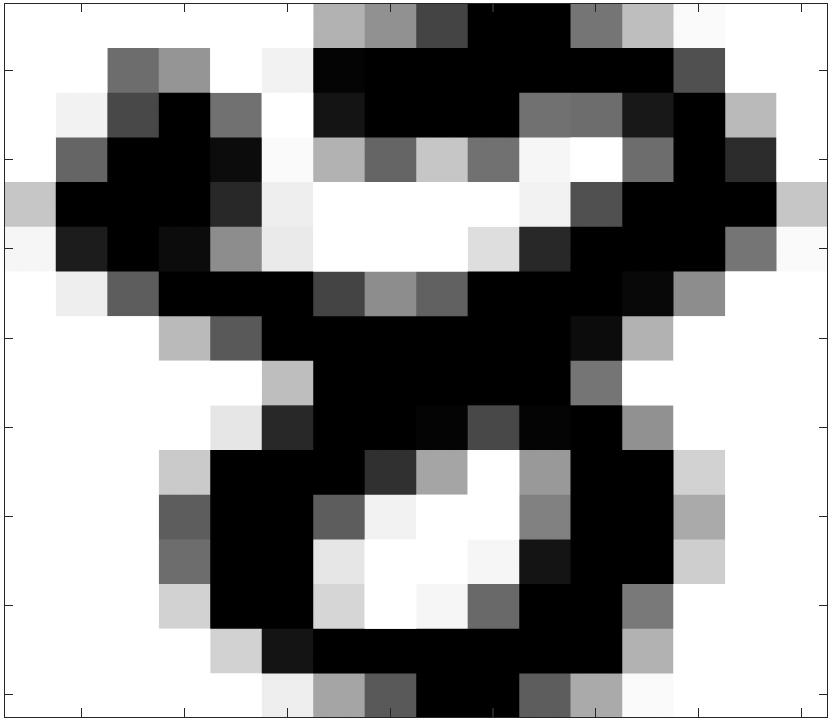}
\includegraphics[scale = .35]{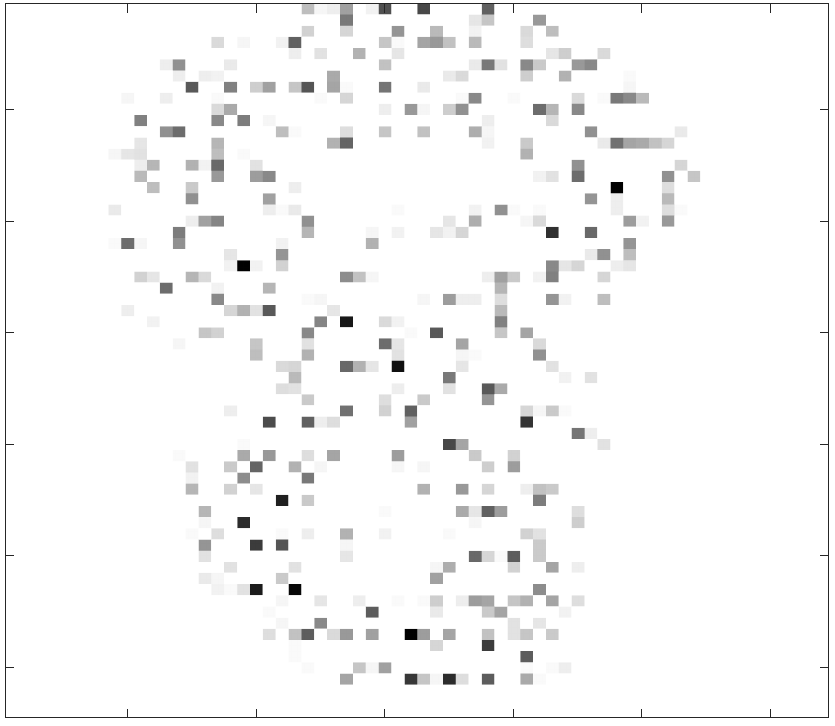}
\end{center}
\caption{Left, an example of digit 8 from the MNIST digits data set. Right, a barycenter computed for four 8's from the data set. It is supported on a four-times-finer grid. }\label{fig:example8}
\end{figure}

We focus our comparison of model sizes on LPs (\ref{baryLP}) and (\ref{LPw}). While both LPs (\ref{barymodLP}) and (\ref{LPhybrid}) provide further improvements on LP (\ref{baryLP}), they come at significant cost: either expert knowledge on the input or computationally expensive preprocessing are required. We move this discussion to Section \ref{sec:apriori}.

So assume that all measures $P_i$ are supported on a $d$-dimensional regular grid $G_{\text{org}}$ of integer step sizes in each direction, each coordinate going from $1$ to $K$. Because the measures are evenly weighted, i.e., $\lambda_i = \frac{1}{n}$ for all $i=1,\dots,n$, $S$ is contained in a $d$-dimensional, regular grid $G_{\text{full}}$ that is $n$-times-finer than the original grid in each of the $d$ directions. We get $|S| \leq |G_{\text{full}}|= (nK-n+1)^d$, as the finer grid only runs between the boundary points. In Figure \ref{fig:example8} (right), we depict a barycenter computed for four evenly weighted digits. It is supported sparsely on a regular, four-times-finer grid; the dimension is $61 \times 61$ instead of $16 \times 16$. (Note $61=4 \cdot 16 - 3$.)

Recall that LP (\ref{baryLP}) uses a variable $z_k$ for each possible support point $\x_k\in S$, and variables $y_{ijk}$ that allow transport from $\x_k$ to each $\x_{ij}$. The number of variables is $$(nK-n+1)^d (1+\sum_{i=1}^n |P_i| )\leq (nK-n+1)^d (1+nK^d) \, \in \, O(n^{d+1}K^{2d})$$ and the number of constraints is $$nK^d +n(nK-n+1)^d \, \in \, O(n^{d+1}K^d).$$ Both the number of variables and the number of constraints are polynomial in $n$ with degree $d+1$. In fact, it now is not hard to see that the Discrete Barycenter Problem with grid-structured data, even weights $\lambda_i$, and input information (coordinates of support points, $K$) on the grid, can be solved in strongly polynomial number by using $G_{\text{full}}$ in place of $S$ in the formulation of LP  (\ref{baryLP}). The proof is similar to the proof of existence of a strongly polynomial $2$-approximation for the Discrete Barycenter Problem in \cite{b-17}.

\begin{theorem}\label{thm:stronggrid}
For all rational data and fixed dimension $d$, the Discrete Barycenter Problem for data supported on a grid $G_{\text{org}}$, $\lambda_i = \frac{1}{n} $ for all $i=1,\dots,n$, and input on $G_{\text{org}}$ itself, can be solved in strongly polynomial time in $n$ and $K$ using LP (\ref{baryLP}), by using  $G_{\text{full}}$ in place of $S$.
\end{theorem}

\begin{proof}{Proof.}
We have to show that setup and solution of LP (\ref{baryLP}) using $G_{\text{full}}$ is strongly polynomial, i.e., it is possible in a polynomial number of arithmetic operations in the number of integers in the input and by performing only computations on numbers that are polynomial in the input.

In the above, we have seen that the problem has a strongly polynomial number of variables and constraints for fixed $d$. The only numbers in the model that might not be explicitly stated in the input are $\|\x_k - \x_{ij} \|^2$. The $ \x_{ij}$ lie in $G_{\text{org}}$, the $\x_k$ in $G_{\text{full}}$; both can be computed efficiently from input on $G_{\text{org}}$. Then each of the strongly polynomially many terms $\|\x_k - \x_{ij} \|^2=(\x_k - \x_{ij})^T(\x_k - \x_{ij})$ only requires the computation of the vector difference $\x_k - \x_{ij}$ and the scalar product of the resulting vector with itself.  Thus, the setup and the size of the LP indeed are strongly polynomial.

Now, recall that linear programs are generally solvable in weakly polynomial time. This indicates that the absolutes of numbers in the input are connected to the number of required arithmetic operations. However, a famous result by Tardos \cite{t-86} states that it suffices to only consider numbers in the constraint matrix, and not in the objective function or right-hand sides. The constraint matrix for LP (\ref{baryLP}) consists only of $0$'s and $1$'s. This implies that the LP can be solved in strongly polynomial time.\eoproof
\end{proof}

 By contrast, recall LP (\ref{LPw}) has up to $(p_{\max})^n$ variables (but only $n p_{\max}$ constraints). If the actual support of the measures is a linear fraction of the number $K^d$ of support points in the whole grid (we observe factors between $\tau=\frac{1}{5}$ and $\tau=\frac{1}{3}$ for the MNIST digits data set) then one obtains an exponential growth in $\Theta((\tau K^d)^n)$. In summary, LP (\ref{baryLP}) is of polynomial size in $n$ for fixed dimension $d$, while LP (\ref{LPw}) is of exponential size in $n$, even if $d$ is fixed. See Table \ref{tab:gridVC} for the numbers for a worst-case scenario, in which each grid point has nonzero mass.

\begin{table}[t]
\begin{center}
\begin{tabular}{|c|c|c|}\hline
LP formulation & Variables & Constraints \\ \hline
(\ref{baryLP}) & $(nK-n+1)^d(1+nK^d)$ & $nK^d+n(nK-n+1)^d$ \\ \hline
(\ref{LPw}) & $(K^d)^n$ & $nK^d$ \\ \hline
\end{tabular}
\end{center}
\caption{The number of variables and constraints for $n$ measures supported densely on a grid of length $K$.}\label{tab:gridVC}
\end{table}

Both LP (\ref{barymodLP}) and (\ref{LPhybrid}) scale polynomially in $n$, too: By Lemma \ref{lem:lessy}, LP (\ref{barymodLP}) always is a strict improvement over LP (\ref{baryLP}) in model size. An even smaller model size can be achieved using LP (\ref{LPhybrid}) through an optimal choice of $(S^*)^y$ and $(S^*)^w$. However, computationally expensive preprocessing is required to set up both of these models. In the next section, we discuss the hardness of this preprocessing in general, as well as some efficient, but less powerful preprocessing routines tailored to grid data. 

\section{Preprocessing}\label{sec:process}

We now analyze the preprocessing that is necessary for each LP from Section \ref{sec:Form}. First, we prove that the setup of all LPs is hard (and thus computationally expensive), even if the resulting LPs are of polynomial size. In fact, a small step of the setup for LPs  (\ref{barymodLP}) and (\ref{LPhybrid}), which may even have to be repeated exponentially many times, already is hard.

Second, we turn to evenly weighted measures supported on a regular grid (as in Section \ref{sec:modgrid}). Here, the possible support set $S$ is of polynomial size, in contrast to the exponential size of $S^*$. Our goal is to reduce the model sizes {\em without} processing $S^*$. LP  (\ref{baryLP}) scales polynomially for fixed dimension $d$ in this setting (c.f.  Table \ref{tab:gridVC}). Thus, the preprocessing has to be fast, while still providing a significant improvement. To achieve this, we present preprocessing routines that use {\em only the grid structure} and simple properties derived from measures $P_1, \ldots, P_n$ to improve computations. The positive impact on computation times is highlighted in Section \ref{sec:comp}.

\subsection{Hardness of General Preprocessing}\label{sec:hard}

If there is no expert knowledge on structure underlying a set of measures $P_1, \ldots, P_n$ (or they are in general position), we first have to process the set $S^*$ in order to set up any of the linear programming formulations. In particular, this is necessary for determining the set $S$, as well as the sets $S_k$ and $S_{ij}$ required for LPs  (\ref{barymodLP}) and (\ref{LPhybrid}). However, the exponential size of $S^*$ is not the only reason why this preprocessing is computationally expensive. We begin by proving that it is already NP-hard to decide whether a given  $\x$ lies in $S$ or not. 

\begin{lemma}\label{lem:goodLPishard}
Let $P_1,\dots,P_n \subset \mathbb{R}^d$ be discrete measures, let $\lambda\in \R^n_{+}$ be a weight vector with $\sum_{i=1}^n \lambda_i=1$, and let $\x\in \mathbb{R}^d$. Then it is NP-hard to decide whether $\x\in S$, even for $d=1$.
\end{lemma}

\begin{proof}{Proof.}
For convenience, we call the decision whether $\x\in S$ for $d=1$ the {\em possible support problem}. We prove the claim by showing that an efficient decision of the possible support problem would lead to an efficient decision of the subset sum problem, which is known to be NP-complete. A variant of the subset sum problem can be stated as follows: given a set of non-zero integers $p_1,\dots,p_n\in \mathbb{Z}$, is there a subset whose sum is a given $s\in \mathbb{Z}$? (We allow for non-proper subsets, the empty set or the whole set, in this formulation to make the following arguments a bit less technical.)

Consider an instance $\mathcal{I}_s$ of the subset sum problem. We construct an instance $\mathcal{I}_p$ of the possible support problem from it as follows: Let $P_1,\dots,P_n \subset \mathbb{Z}$ be discrete measures, where each $P_i$ consists of the two support points $p_i$ and $0$ (both of mass $\frac{1}{2}$, but this is not relevant). Further, let $\lambda_1,\dots,\lambda_n=\frac{1}{n}$ and $\x=\frac{1}{n}\cdot s$. Clearly, this construction is polynomial. It remains to prove that $\mathcal{I}_s$ is a yes-instance if and only if $\mathcal{I}_p$ is a yes-instance. 

If $\mathcal{I}_p$ is a yes-instance, then $\x=\frac{1}{n}\cdot s= \sum_{i=1}^n \frac{1}{n} \x_{ij}$, where $\x_{ij}$ is either $p_i$ or $0$ for all $i\leq n$. Equivalently $s = n\cdot \x= \sum_{i=1}^n \x_{ij}$, where $\x_{ij}$ is either $p_i$ or $0$ for all $i\leq n$. Thus there exists a subset of the $p_1,\dots,p_n$ adding up to $s$, so $\mathcal{I}_s$ is a yes-instance, too. 

Conversely, let $\mathcal{I}_s$ be a yes-instance, i.e., there exists a set of indices $I\subset \{1,\dots,n\}$ such that $s=\sum\limits_{i\in I} p_i$. Then $\mathcal{I}_p$ is a yes-instance, too, as $\x=\frac{1}{n}\cdot s= \sum_{i=1}^n \frac{1}{n} \x_{ij}$ for $\x_{ij}=p_i$ for all $i\in I$, and $0$ else. This proves the claim.\eoproof
\end{proof}

Lemma \ref{lem:goodLPishard} has several implications. First, it transfers to the set $S$ as a whole.

\begin{corollary}\label{cor:goodLPishard2}
Let $P_1,\dots,P_n \subset \mathbb{R}^d$ be discrete measures and let $\lambda\in \R^n_{+}$ be a weight vector with $\sum_{i=1}^n \lambda_i=1$. Then it is NP-hard to decide whether a given set $S\subset \mathbb{R}^d$ is the set of possible support points.
\end{corollary}

This implies that it is (computationally) hard to construct $S$ from the input, {\em even if it is of polynomial size}. Further, Lemma \ref{lem:goodLPishard} transfers to the hardness of constructing $S_k$ for a fixed $\x_k\in S$. 
\begin{corollary}\label{cor:goodLPishard}
Let $P_1,\dots,P_n \subset \mathbb{R}^d$ be discrete measures, let $\lambda\in \R^n_{+}$ be a weight vector with $\sum_{i=1}^n \lambda_i=1$, and let $\x_k\in S$. Then it is NP-hard to decide whether a given set $S_k$ is correct.
\end{corollary}
Recall that $S_k$ contains the pairs $(i,j)$ of the $\x_{ij}$ from which a particular $\x_k$ can be constructed. The ability to construct $S_k$ for a given $\x_k$ in polynomial time would be sufficient to decide whether a given $S_k$ is correct in polynomial time, too -- a contradiction to Corollary \ref{cor:goodLPishard}, unless $P=NP$. 

\subsection{Preprocessing for Grid-structured Data}\label{sec:apriori}

The main benefit of evenly weighted grid-structured data, i.e., $n$ measures supported in a regular grid $G_{\text{org}}$ and $\lambda_i=\frac{1}{n}$, is that the set $S$ is contained in an $n$-times-finer grid $G_{\text{full}}$. This allows the efficient setup and solution of the Discrete Barycenter Problem, by using $G_{\text{full}}$ instead of $S$ in LP (\ref{baryLP}); recall Theorem \ref{thm:stronggrid} and see Table \ref{tab:gridVC} for exact model sizes. However, this approach can still be improved significantly. To this end, an exact construction of LPs  (\ref{barymodLP}) and (\ref{LPhybrid}) may be impractical: It requires the processing of the exponential-sized set $S^*$, and each step of this process is hard by itself (Lemma \ref{lem:goodLPishard}). 

In this section, we have a different goal. We devise some efficient preprocessing routines, only using the grid structure and simple information on $P_1, \ldots, P_n$, to reduce the model size. But to devise these routines, we first have to understand what happens in an exact preprocessing. For simplicity, we assume that each of the $P_i$ has mass on the whole grid $G_{\text{org}}$. In this case, $S=G_{\text{full}}$.

\noindent{\bf Preprocessing for LP (\ref{barymodLP}).} To set up LP (\ref{barymodLP}), the sets $S_{ij}$ and $S_k$ are constructed during the processing of $S^*$. They are used to reduce the number of $y$-variables for each $\x_k\in S$. 

The maximum size of $S_k$, respectively the maximum number of $y$-variables, is $nK^d$ for each $\x_k$. The actual size of $S_k$ depends on the position of $\x_k$ in the grid. An example is depicted in Figure \ref{fig:outergrid}, which shows the effect for four measures in $\mathbb{R}^2$ in a regular grid with $K=4$, resulting in a $13\times 13$ grid. For any $\x_k \in G_{\text{full}}$ with a coordinate in the outer $K-1$ rows or columns (here, 3 rows and columns, highlighted), there exist $\x_{ij} \in G_{\text{org}}$ such that no weighted mean involving that $\x_{ij}$ gives the $\x_k$. Therefore, the corresponding variables $y_{ijk}$ are not introduced in the preprocessing. All other points -- those in the `center' of the grid $G$ -- require all $nK^d$ $y$-variables.

\begin{figure}[t]
\begin{center}
\begin{tikzpicture}
\node (11) at (1,2.7) {$(1,1)$};
\node (41) at (4,2.7) {$(4,1)$};
\node (14) at (1,6.3) {$(1,4)$};
\node (44) at (4,6.3) {$(4,4)$};

\fill [black] (1,3) circle (2pt);
\fill [black] (1.25,3) circle (2pt);
\fill [black] (1.5,3) circle (2pt);
\fill [black] (1.75,3) circle (2pt);
\fill [black] (2,3) circle (2pt);
\fill [black] (2.25,3) circle (2pt);
\fill [black] (2.5,3) circle (2pt);
\fill [black] (2.75,3) circle (2pt);
\fill [black] (3,3) circle (2pt);
\fill [black] (3.25,3) circle (2pt);
\fill [black] (3.5,3) circle (2pt);
\fill [black] (3.75,3) circle (2pt);
\fill [black] (4,3) circle (2pt);
\fill [black] (1,3.25) circle (2pt);
\fill [black] (1.25,3.25) circle (2pt);
\fill [black] (1.5,3.25) circle (2pt);
\fill [black] (1.75,3.25) circle (2pt);
\fill [black] (2,3.25) circle (2pt);
\fill [black] (2.25,3.25) circle (2pt);
\fill [black] (2.5,3.25) circle (2pt);
\fill [black] (2.75,3.25) circle (2pt);
\fill [black] (3,3.25) circle (2pt);
\fill [black] (3.25,3.25) circle (2pt);
\fill [black] (3.5,3.25) circle (2pt);
\fill [black] (3.75,3.25) circle (2pt);
\fill [black] (4,3.25) circle (2pt);

\fill [black] (1,3.5) circle (2pt);
\fill [black] (1.25,3.5) circle (2pt);
\fill [black] (1.5,3.5) circle (2pt);
\fill [black] (1.75,3.5) circle (2pt);
\fill [black] (2,3.5) circle (2pt);
\fill [black] (2.25,3.5) circle (2pt);
\fill [black] (2.5,3.5) circle (2pt);
\fill [black] (2.75,3.5) circle (2pt);
\fill [black] (3,3.5) circle (2pt);
\fill [black] (3.25,3.5) circle (2pt);
\fill [black] (3.5,3.5) circle (2pt);
\fill [black] (3.75,3.5) circle (2pt);
\fill [black] (4,3.5) circle (2pt);
\fill [black] (1,3.75) circle (2pt);
\fill [black] (1.25,3.75) circle (2pt);
\fill [black] (1.5,3.75) circle (2pt);
\fill [black] (1.75,3.75) circle (2pt);
\fill [black] (2,3.75) circle (2pt);
\fill [black] (2.25,3.75) circle (2pt);
\fill [black] (2.5,3.75) circle (2pt);
\fill [black] (2.75,3.75) circle (2pt);
\fill [black] (3,3.75) circle (2pt);
\fill [black] (3.25,3.75) circle (2pt);
\fill [black] (3.5,3.75) circle (2pt);
\fill [black] (3.75,3.75) circle (2pt);
\fill [black] (4,3.75) circle (2pt);

\fill [black] (1,4) circle (2pt);
\fill [black] (1.25,4) circle (2pt);
\fill [black] (1.5,4) circle (2pt);
\fill [black] (1.75,4) circle (2pt);
\fill [black] (2,4) circle (2pt);
\fill [black] (2.25,4) circle (2pt);
\fill [black] (2.5,4) circle (2pt);
\fill [black] (2.75,4) circle (2pt);
\fill [black] (3, 4) circle (2pt);
\fill [black] (3.25, 4) circle (2pt);
\fill [black] (3.5, 4) circle (2pt);
\fill [black] (3.75, 4) circle (2pt);
\fill [black] (4, 4) circle (2pt);
\fill [black] (1,4.25) circle (2pt);
\fill [black] (1.25,4.25) circle (2pt);
\fill [black] (1.5,4.25) circle (2pt);
\fill [black] (1.75,4.25) circle (2pt);
\fill [black] (2,4.25) circle (2pt);
\fill [black] (2.25,4.25) circle (2pt);
\fill [black] (2.5,4.25) circle (2pt);
\fill [black] (2.75,4.25) circle (2pt);
\fill [black] (3, 4.25) circle (2pt);
\fill [black] (3.25, 4.25) circle (2pt);
\fill [black] (3.5, 4.25) circle (2pt);
\fill [black] (3.75, 4.25) circle (2pt);
\fill [black] (4, 4.25) circle (2pt);
\fill [black] (1,4.5) circle (2pt);
\fill [black] (1.25,4.5) circle (2pt);
\fill [black] (1.5,4.5) circle (2pt);
\fill [black] (1.75,4.5) circle (2pt);
\fill [black] (2,4.5) circle (2pt);
\fill [black] (2.25,4.5) circle (2pt);
\fill [black] (2.5,4.5) circle (2pt);
\fill [black] (2.75,4.5) circle (2pt);
\fill [black] (3, 4.5) circle (2pt);
\fill [black] (3.25, 4.5) circle (2pt);
\fill [black] (3.5, 4.5) circle (2pt);
\fill [black] (3.75, 4.5) circle (2pt);
\fill [black] (4, 4.5) circle (2pt);
\fill [black] (1,4.75) circle (2pt);
\fill [black] (1.25,4.75) circle (2pt);
\fill [black] (1.5,4.75) circle (2pt);
\fill [black] (1.75,4.75) circle (2pt);
\fill [black] (2,4.75) circle (2pt);
\fill [black] (2.25,4.75) circle (2pt);
\fill [black] (2.5,4.75) circle (2pt);
\fill [black] (2.75,4.75) circle (2pt);
\fill [black] (3, 4.75) circle (2pt);
\fill [black] (3.25, 4.75) circle (2pt);
\fill [black] (3.5, 4.75) circle (2pt);
\fill [black] (3.75, 4.75) circle (2pt);
\fill [black] (4, 4.75) circle (2pt);

\fill [black] (1,5) circle (2pt);
\fill [black] (1.25,5) circle (2pt);
\fill [black] (1.5,5) circle (2pt);
\fill [black] (1.75,5) circle (2pt);
\fill [black] (2,5) circle (2pt);
\fill [black] (2.25,5) circle (2pt);
\fill [black] (2.5,5) circle (2pt);
\fill [black] (2.75,5) circle (2pt);
\fill [black] (3, 5) circle (2pt);
\fill [black] (3.25, 5) circle (2pt);
\fill [black] (3.5, 5) circle (2pt);
\fill [black] (3.75, 5) circle (2pt);
\fill [black] (4, 5) circle (2pt);
\fill [black] (1,5.25) circle (2pt);
\fill [black] (1.25,5.25) circle (2pt);
\fill [black] (1.5,5.25) circle (2pt);
\fill [black] (1.75,5.25) circle (2pt);
\fill [black] (2,5.25) circle (2pt);
\fill [black] (2.25,5.25) circle (2pt);
\fill [black] (2.5,5.25) circle (2pt);
\fill [black] (2.75,5.25) circle (2pt);
\fill [black] (3, 5.25) circle (2pt);
\fill [black] (3.25, 5.25) circle (2pt);
\fill [black] (3.5, 5.25) circle (2pt);
\fill [black] (3.75, 5.25) circle (2pt);
\fill [black] (4, 5.25) circle (2pt);
\fill [black] (1,5.5) circle (2pt);
\fill [black] (1.25,5.5) circle (2pt);
\fill [black] (1.5,5.5) circle (2pt);
\fill [black] (1.75,5.5) circle (2pt);
\fill [black] (2,5.5) circle (2pt);
\fill [black] (2.25,5.5) circle (2pt);
\fill [black] (2.5,5.5) circle (2pt);
\fill [black] (2.75,5.5) circle (2pt);
\fill [black] (3, 5.5) circle (2pt);
\fill [black] (3.25, 5.5) circle (2pt);
\fill [black] (3.5, 5.5) circle (2pt);
\fill [black] (3.75, 5.5) circle (2pt);
\fill [black] (4, 5.5) circle (2pt);
\fill [black] (1,5.75) circle (2pt);
\fill [black] (1.25,5.75) circle (2pt);
\fill [black] (1.5,5.75) circle (2pt);
\fill [black] (1.75,5.75) circle (2pt);
\fill [black] (2,5.75) circle (2pt);
\fill [black] (2.25,5.75) circle (2pt);
\fill [black] (2.5,5.75) circle (2pt);
\fill [black] (2.75,5.75) circle (2pt);
\fill [black] (3, 5.75) circle (2pt);
\fill [black] (3.25, 5.75) circle (2pt);
\fill [black] (3.5, 5.75) circle (2pt);
\fill [black] (3.75, 5.75) circle (2pt);
\fill [black] (4, 5.75) circle (2pt);
\fill [black] (1,6) circle (2pt);
\fill [black] (1.25,6) circle (2pt);
\fill [black] (1.5,6) circle (2pt);
\fill [black] (1.75,6) circle (2pt);
\fill [black] (2,6) circle (2pt);
\fill [black] (2.25,6) circle (2pt);
\fill [black] (2.5,6) circle (2pt);
\fill [black] (2.75,6) circle (2pt);
\fill [black] (3, 6) circle (2pt);
\fill [black] (3.25, 6) circle (2pt);
\fill [black] (3.5, 6) circle (2pt);
\fill [black] (3.75, 6) circle (2pt);
\fill [black] (4, 6) circle (2pt);
\end{tikzpicture}
\qquad \qquad \qquad \qquad \qquad
\begin{tikzpicture}
\node (11) at (1,2.7) {$(1,1)$};
\node (41) at (4,2.7) {$(4,1)$};
\node (14) at (1,6.3) {$(1,4)$};
\node (44) at (4,6.3) {$(4,4)$};

\fill [green] (1,3) circle (2pt);
\fill [green] (1.25,3) circle (2pt);
\fill [green] (1.5,3) circle (2pt);
\fill [green] (1.75,3) circle (2pt);
\fill [green] (2,3) circle (2pt);
\fill [green] (2.25,3) circle (2pt);
\fill [green] (2.5,3) circle (2pt);
\fill [green] (2.75,3) circle (2pt);
\fill [green] (3,3) circle (2pt);
\fill [green] (3.25,3) circle (2pt);
\fill [green] (3.5,3) circle (2pt);
\fill [green] (3.75,3) circle (2pt);
\fill [green] (4,3) circle (2pt);

\fill [green] (1,3.25) circle (2pt);
\fill [green] (1.25,3.25) circle (2pt);
\fill [green] (1.5,3.25) circle (2pt);
\fill [green] (1.75,3.25) circle (2pt);
\fill [green] (2,3.25) circle (2pt);
\fill [green] (2.25,3.25) circle (2pt);
\fill [green] (2.5,3.25) circle (2pt);
\fill [green] (2.75,3.25) circle (2pt);
\fill [green] (3,3.25) circle (2pt);
\fill [green] (3.25,3.25) circle (2pt);
\fill [green] (3.5,3.25) circle (2pt);
\fill [green] (3.75,3.25) circle (2pt);
\fill [green] (4,3.25) circle (2pt);
\fill [green] (1,3.5) circle (2pt);
\fill [green] (1.25,3.5) circle (2pt);
\fill [green] (1.5,3.5) circle (2pt);
\fill [green] (1.75,3.5) circle (2pt);
\fill [green] (2,3.5) circle (2pt);
\fill [green] (2.25,3.5) circle (2pt);
\fill [green] (2.5,3.5) circle (2pt);
\fill [green] (2.75,3.5) circle (2pt);
\fill [green] (3,3.5) circle (2pt);
\fill [green] (3.25,3.5) circle (2pt);
\fill [green] (3.5,3.5) circle (2pt);
\fill [green] (3.75,3.5) circle (2pt);
\fill [green] (4,3.5) circle (2pt);
\fill [green] (1,3.75) circle (2pt);
\fill [green] (1.25,3.75) circle (2pt);
\fill [green] (1.5,3.75) circle (2pt);
\fill [black] (1.75,3.75) circle (2pt);
\fill [black] (2,3.75) circle (2pt);
\fill [black] (2.25,3.75) circle (2pt);
\fill [black] (2.5,3.75) circle (2pt);
\fill [black] (2.75,3.75) circle (2pt);
\fill [black] (3,3.75) circle (2pt);
\fill [black] (3.25,3.75) circle (2pt);
\fill [green] (3.5,3.75) circle (2pt);
\fill [green] (3.75,3.75) circle (2pt);
\fill [green] (4,3.75) circle (2pt);

\fill [green] (1,4) circle (2pt);
\fill [green] (1.25,4) circle (2pt);
\fill [green] (1.5,4) circle (2pt);
\fill [black] (1.75,4) circle (2pt);
\fill [black] (2,4) circle (2pt);
\fill [black] (2.25,4) circle (2pt);
\fill [black] (2.5,4) circle (2pt);
\fill [black] (2.75,4) circle (2pt);
\fill [black] (3, 4) circle (2pt);
\fill [black] (3.25, 4) circle (2pt);
\fill [green] (3.5,4) circle (2pt);
\fill [green] (3.75,4) circle (2pt);
\fill [green] (4,4) circle (2pt);

\fill [green] (1,4.25) circle (2pt);
\fill [green] (1.25,4.25) circle (2pt);
\fill [green] (1.5,4.25) circle (2pt);
\fill [black] (1.75,4.25) circle (2pt);
\fill [black] (2,4.25) circle (2pt);
\fill [black] (2.25,4.25) circle (2pt);
\fill [black] (2.5,4.25) circle (2pt);
\fill [black] (2.75,4.25) circle (2pt);
\fill [black] (3, 4.25) circle (2pt);
\fill [black] (3.25, 4.25) circle (2pt);
\fill [green] (3.5,4.25) circle (2pt);
\fill [green] (3.75,4.25) circle (2pt);
\fill [green] (4,4.25) circle (2pt);

\fill [green] (1,4.5) circle (2pt);
\fill [green] (1.25,4.5) circle (2pt);
\fill [green] (1.5,4.5) circle (2pt);
\fill [black] (1.75,4.5) circle (2pt);
\fill [black] (2,4.5) circle (2pt);
\fill [black] (2.25,4.5) circle (2pt);
\fill [black] (2.5,4.5) circle (2pt);
\fill [black] (2.75,4.5) circle (2pt);
\fill [black] (3, 4.5) circle (2pt);
\fill [black] (3.25, 4.5) circle (2pt);
\fill [green] (3.5,4.5) circle (2pt);
\fill [green] (3.75,4.5) circle (2pt);
\fill [green] (4,4.5) circle (2pt);

\fill [green] (1,4.75) circle (2pt);
\fill [green] (1.25,4.75) circle (2pt);
\fill [green] (1.5,4.75) circle (2pt);
\fill [black] (1.75,4.75) circle (2pt);
\fill [black] (2,4.75) circle (2pt);
\fill [black] (2.25,4.75) circle (2pt);
\fill [black] (2.5,4.75) circle (2pt);
\fill [black] (2.75,4.75) circle (2pt);
\fill [black] (3, 4.75) circle (2pt);
\fill [black] (3.25, 4.75) circle (2pt);
\fill [green] (3.5,4.75) circle (2pt);
\fill [green] (3.75,4.75) circle (2pt);
\fill [green] (4,4.75) circle (2pt);

\fill [green] (1,5) circle (2pt);
\fill [green] (1.25,5) circle (2pt);
\fill [green] (1.5,5) circle (2pt);
\fill [black] (1.75,5) circle (2pt);
\fill [black] (2,5) circle (2pt);
\fill [black] (2.25,5) circle (2pt);
\fill [black] (2.5,5) circle (2pt);
\fill [black] (2.75,5) circle (2pt);
\fill [black] (3, 5) circle (2pt);
\fill [black] (3.25, 5) circle (2pt);
\fill [green] (3.5, 5) circle (2pt);
\fill [green] (3.75, 5) circle (2pt);
\fill [green] (4, 5) circle (2pt);
\fill [green] (1,5.25) circle (2pt);
\fill [green] (1.25,5.25) circle (2pt);
\fill [green] (1.5,5.25) circle (2pt);
\fill [black] (1.75,5.25) circle (2pt);
\fill [black] (2,5.25) circle (2pt);
\fill [black] (2.25,5.25) circle (2pt);
\fill [black] (2.5,5.25) circle (2pt);
\fill [black] (2.75,5.25) circle (2pt);
\fill [black] (3, 5.25) circle (2pt);
\fill [black] (3.25, 5.25) circle (2pt);
\fill [green] (3.5, 5.25) circle (2pt);
\fill [green] (3.75, 5.25) circle (2pt);
\fill [green] (4, 5.25) circle (2pt);
\fill [green] (1,5.5) circle (2pt);
\fill [green] (1.25,5.5) circle (2pt);
\fill [green] (1.5,5.5) circle (2pt);
\fill [green] (1.75,5.5) circle (2pt);
\fill [green] (2,5.5) circle (2pt);
\fill [green] (2.25,5.5) circle (2pt);
\fill [green] (2.5,5.5) circle (2pt);
\fill [green] (2.75,5.5) circle (2pt);
\fill [green] (3, 5.5) circle (2pt);
\fill [green] (3.25, 5.5) circle (2pt);
\fill [green] (3.5, 5.5) circle (2pt);
\fill [green] (3.75, 5.5) circle (2pt);
\fill [green] (4, 5.5) circle (2pt);
\fill [green] (1,5.75) circle (2pt);
\fill [green] (1.25,5.75) circle (2pt);
\fill [green] (1.5,5.75) circle (2pt);
\fill [green] (1.75,5.75) circle (2pt);
\fill [green] (2,5.75) circle (2pt);
\fill [green] (2.25,5.75) circle (2pt);
\fill [green] (2.5,5.75) circle (2pt);
\fill [green] (2.75,5.75) circle (2pt);
\fill [green] (3, 5.75) circle (2pt);
\fill [green] (3.25, 5.75) circle (2pt);
\fill [green] (3.5, 5.75) circle (2pt);
\fill [green] (3.75, 5.75) circle (2pt);
\fill [green] (4, 5.75) circle (2pt);
\fill [green] (1,6) circle (2pt);
\fill [green] (1.25,6) circle (2pt);
\fill [green] (1.5,6) circle (2pt);
\fill [green] (1.75,6) circle (2pt);
\fill [green] (2,6) circle (2pt);
\fill [green] (2.25,6) circle (2pt);
\fill [green] (2.5,6) circle (2pt);
\fill [green] (2.75,6) circle (2pt);
\fill [green] (3, 6) circle (2pt);
\fill [green] (3.25, 6) circle (2pt);
\fill [green] (3.5, 6) circle (2pt);
\fill [green] (3.75, 6) circle (2pt);
\fill [green] (4, 6) circle (2pt);
\end{tikzpicture}
\end{center}
\caption{The barycenter of $4$ measures densely supported in $4\times 4$ grids lies in a $13 \times 13$ grid. The outer $3$ rows and columns (in green) have points in each original $4 \times 4$ grid to which they will not transport.}\label{fig:outergrid}
\end{figure}
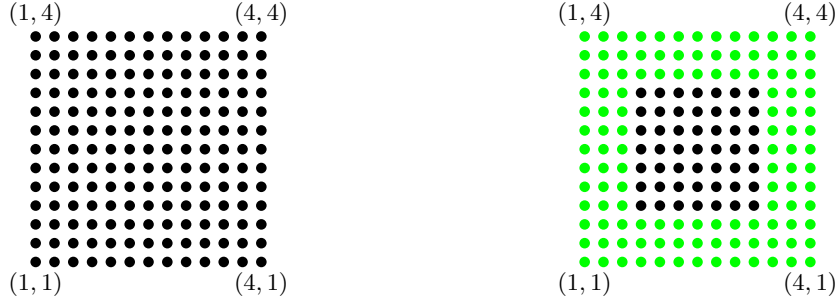

\noindent{\bf Preprocessing for LP (\ref{LPhybrid}).} Next, we discuss the choice between a representation of a weighted mean through $w$-variables versus $y,z$-variables. We choose $y,z$-variables when there are more than $nK^d+1$ combinations producing an $\x_k$. 

We can calculate the number of combinations for each support point $\x_k$ by counting the number of combinations of integers with a specific sum as follows. Beginning in dimension one, denote $\x_{ij}$ as integers between $1$ and $K$ and $s = n\cdot \x_k = \sum_{i=1}^n \x_{ij}$. The number of combinations of $\x_{ij}$ that give $\x_k$ is equivalent to the number $F(s,K,n)$ of combinations of $n$ integers between $1$ and $K$ that sum up to $s$. This is a standard counting problem, in which $n$ $K$-sided dice are tossed and the number of combinations that give sum $s$ are calculated. Thus, $F(s,K,n)$ is given by
\begin{equation*} F(s,K,n) =\sum_{m=0}^n (-1)^m {n \choose m}{ s-mK-1 \choose n-1}. \end{equation*}

We extend this formula to $\x_k\in\mathbb{R}^d$ for higher dimension $d$ by considering each coordinate of $\x_k$ individually. Let $s_l$ be the corresponding sum of the $l$-th coordinate of $\x_k$; then the total number of combinations whose weighted mean is $\x_k$ is \begin{equation*}N_k = \prod_{l = 1}^d F(s_l,K,n).\end{equation*}

So, for any $k$ such that $N_k$ exceeds $nK^d+1$, using $y,z$-variables (rather than introducing individual $w_h, h = 1,\dots, N_k$ for all combinations) produces fewer total variables in LP (\ref{LPhybrid}). The $\x_k$ for which $w$-variables are preferable correspond to the `rounded' corners of $G_{\text{full}}$ as depicted in Figure \ref{fig:corner} -- the darker a support point, the fewer combinations in $S^*$ correspond to it.

Except for small $n$, the majority of the grid prefers $y,z$-variables. Further, the corners where $w$-variables are preferable are contained in the areas of the grid where the number of $y,z$-variables are already reduced from $nK^d+1$ for setup of LP (\ref{barymodLP}); see Figure \ref{fig:outergrid}. In the computational experiments in Section \ref{sec:noprior}, we see that LP (\ref{barymodLP}) provides a significant improvement over LP (\ref{baryLP}). On the other hand, the difference between LP  (\ref{LPhybrid}) and LP (\ref{barymodLP}) is small in view of the computationally more expensive preprocessing needed to set up LP  (\ref{LPhybrid}). In Section \ref{sec:besthybrid}, we provide a different type of data for an example in which the use of LP  (\ref{LPhybrid}) greatly outperforms the others.

\begin{figure}[t]
\begin{center}
\includegraphics[scale = .4]{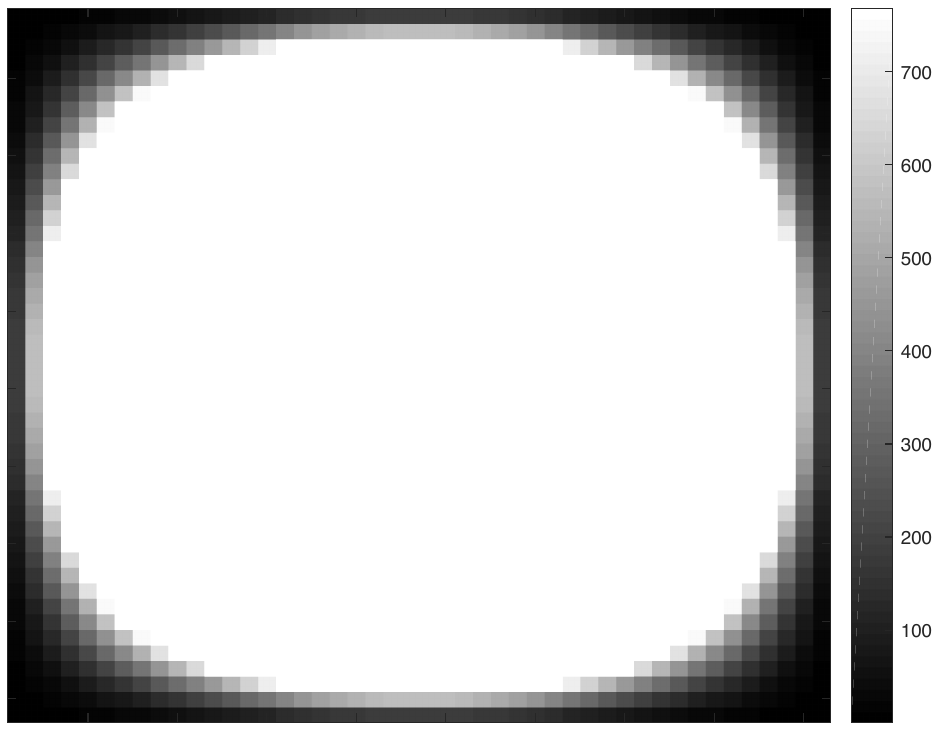}
\includegraphics[scale = .4]{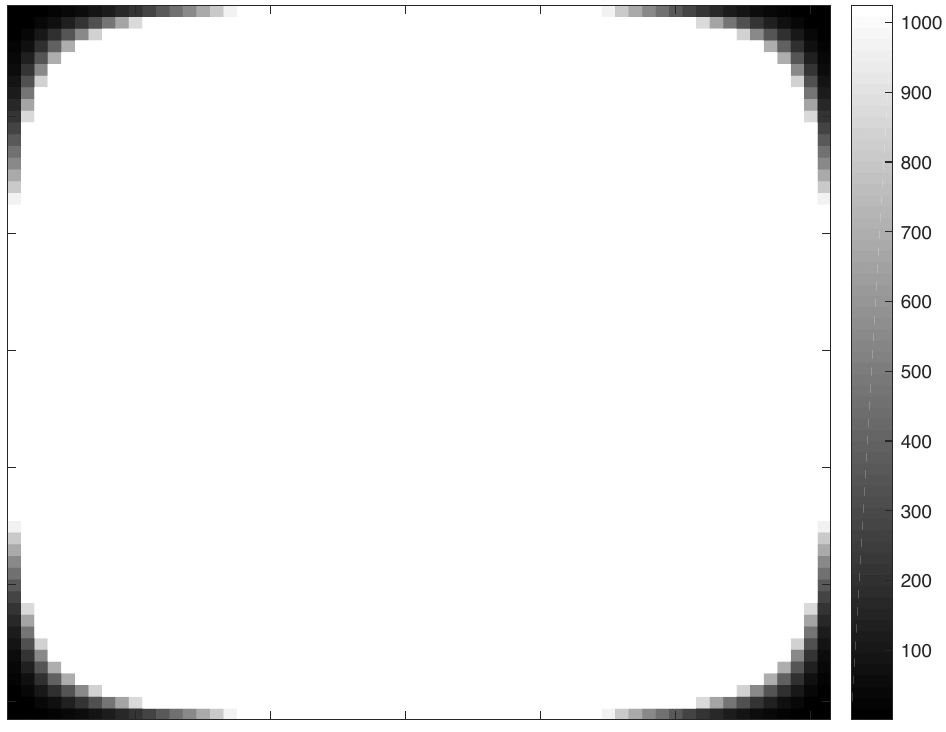}
\end{center}
\caption{The set $S$ of support points for $3$ (left) and $4$ (right) measures supported densely in a $16 \times 16$ grid. The dark support points indicate where $w$-variables are preferable over $y,z$-variables in LP (\ref{LPhybrid}).}\label{fig:corner}
\end{figure}

\noindent{\bf Efficient preprocessing to improve on LP (\ref{baryLP}).} The first step towards a good model for grid-structured data is the implementation of LP (\ref{baryLP}) on $G_{\text{full}}$; recall Theorem \ref{thm:stronggrid}. We now devise another improvement in the form of a smaller subgrid $G \subseteq G_{\text{full}}$ that still contains $S$.

For grid-structured data, $G_{\text{full}}$ may still contain points that are not in $S$. This occurs in the digits data set, because mass only exists on part of the $16 \times 16$ grid. For efficiency, $G_{\text{full}}$ must be generated without checking if the $\x_{ij}$ which produce $\x_k$ have positive mass. Instead, to obtain a smaller sub-grid $G$, we track the largest and smallest values of each coordinate of the dimension for each measure. The weighted means of these extreme coordinate values give sufficient minimum and maximum values for a subgrid $G$ of $G_{\text{full}}$ that contains $S$. We use this set $G$ in LPs (\ref{baryLP}) and LP (\ref{barymodLP}) in the computations for Section \ref{sec:prior}. 

LP (\ref{baryLP}) can be implemented directly on $G$. We have already seen that a significant reduction in model size when working over $S$ is possible through the setup of LP (\ref{barymodLP}). This effect also holds on $G$, but there is an additional challenge: For efficient preprocessing, we have to avoid processing $S^*$ and thus we do not construct the sets $S_{ij}$ and $S_k$ required to set up LP (\ref{barymodLP}) exactly. We need a different approach to eliminate extraneous $y$-variables for $G$.

Recall the `boundary' effect depicted in Figure \ref{fig:outergrid}. The grid points $\x_k$ that have at least one coordinate within the minimum or maximum among all points in $G$ have $y$-variables that can be eliminated -- there exist points $\x_{ij}$ in the original grid $G_{\text{org}}$ such that $y_{ijk} = 0$ in all optimal solutions. Because of the grid structure, it is easy to identify some of these $\x_{ij}$: for example, let $\x_k$ have the largest first coordinate. Then it can only be constructed from the $\x_{ij}$ with a largest first coordinate in each measure. For a second-largest first coordinate, only the second-largest or largest coordinates of $\x_{ij}$ are possible, and so on. 

A reduction in $y$-variables based on this principle requires only a few comparisons for each $\x_k$ and guarantees the elimination of many variables: the outer $K-1$ rows and columns have at least $n K$ possible $\x_{ij}$ to which they do not transport. Thus, the reduction in model size scales with both the width $K$ of the original grid and the number $n$ of measures. In Section \ref{sec:prior}, we see that this reduction is significant in practice and produces the fastest total running times in our computations.

\section{Computational Experiments}\label{sec:comp}

In this section, we compare the performance of the different LPs for different types of data through some computational experiments. For each of the models LP (\ref{LPw}), LP (\ref{barymodLP}), and LP (\ref{LPhybrid}), there exists data such that they become the respectively smallest model with fastest total running times.

Section \ref{sec:compgeneral} is dedicated to data in general position. In Section \ref{sec:digitcomp}, we study computations for grid-based data, and use these to highlight the advantages of the preprocessing routines described in Section \ref{sec:apriori}.  Finally, Section \ref{sec:besthybrid} is dedicated to an example where using LP (\ref{LPhybrid}) significantly outperforms the other formulations.

All computations were run on a standard laptop (MacBook Pro, 2.9 GHz Intel Core i7, 16 GB of RAM, SSD). Data processing and the setup of the LPs, in particular the preprocessing for grid-structured data from Section \ref{sec:apriori}, were implemented in C\texttt{++} and the LPs were solved using Gurobi 8.0. In all computations, the Gurobi solver parameters are fixed and the pre-solvers switched off. We report on solution times for a primal Simplex algorithm -- all single-thread LP solvers behaved similarly. The pre-solving done by LP solvers for general linear programs and our tailored preprocessing routines (Section \ref{sec:apriori}) are conceptually different and do not compete with each other. In practice, we would use both: first our tailored preprocessing to obtain an improved model that is passed to the solver, then the internal pre-solving, and finally the actual run of an LP algorithm. However, the impact of pre-solvers on computation times may vary dramatically between different instances of the same model (sometimes even multiple runs of the same instance). Due to the difficulty of measuring this impact, we follow the standard practice to switch off pre-solvers for a comparison of computation times.

\subsection{Data in General Position: Denver Crime}\label{sec:compgeneral}

We demonstrate the practical advantage of LP (\ref{LPw}) over LP  (\ref{baryLP}) for data in general position through some proof-of-concept computations on crime data for Denver County, which is available as part of the Denver Open Data Catalog (\url{www.denvergov.org/opendata}). The data set is constructed from the dates and locations of murders with a year. Each month forms a measure $P_i, i = 1, \ldots, 12,$ by weighting each murder evenly during the month. Thus, for months with fewer total incidents, the mass on each location is larger. In doing so, each month represents a {\em crime pattern} -- a set of geographical locations where incidents happen `at the same time'. There is no (obvious) underlying structure to the crime patterns, so we obtain a data set in general position.

We are interested in computing locations for police presence such that a fastest average response time to all incidents occurring in all crime patterns is achieved. The police presence itself should be fixed, {\em not} vary depending on the month. A discrete barycenter for this data set can be readily interpreted (and justified) to indicate locations for police presence: The locations are chosen such that a fast response to incidents in each crime pattern is possible. Conversely, for each incident, there is police presence at a location from which a fast response is possible. Recall that a barycenter computation uses squared Euclidean distances; this leads to a fair treatment of outlier incidents. An aggregate image of all murder locations in 2016 and a corresponding barycenter are displayed in Figure \ref{fig:murders}. The radii of the shapes in both parts of the figure are relative to their masses. 

Based on the theoretical analysis in Section \ref{sec:datageneral}, we only highlight the advantage of LP (\ref{LPw}) over LP (\ref{baryLP}): While LP (\ref{barymodLP}) would also be an improvement over LP (\ref{baryLP}), it is dominated by the even better performance of LP (\ref{LPw}). Further, a best implementation of LP (\ref{LPhybrid}) produces the same linear program LP (\ref{LPw}).  

In Table \ref{GenTotalRun}, we compare the sizes of and running times for LPs (\ref{baryLP}) and (\ref{LPw}) for five years, 2012 to 2016. The reduction in number of variables and number of constraints is substantial; the linear growth of the number of constraints is readily visible in the small number of rows, and the number of variables is typically about 2\% of the number for LP (\ref{baryLP}). For two years, 2013 and 2015, LP (\ref{baryLP}) is unable to solve due to memory capacity, so, in particular, the use of LP (\ref{LPw}) allows the computation of new solutions. The total running times using LP (\ref{LPw}) are typically only a few seconds. 

\begin{figure}[t]
\begin{center}
\includegraphics[scale = .4]{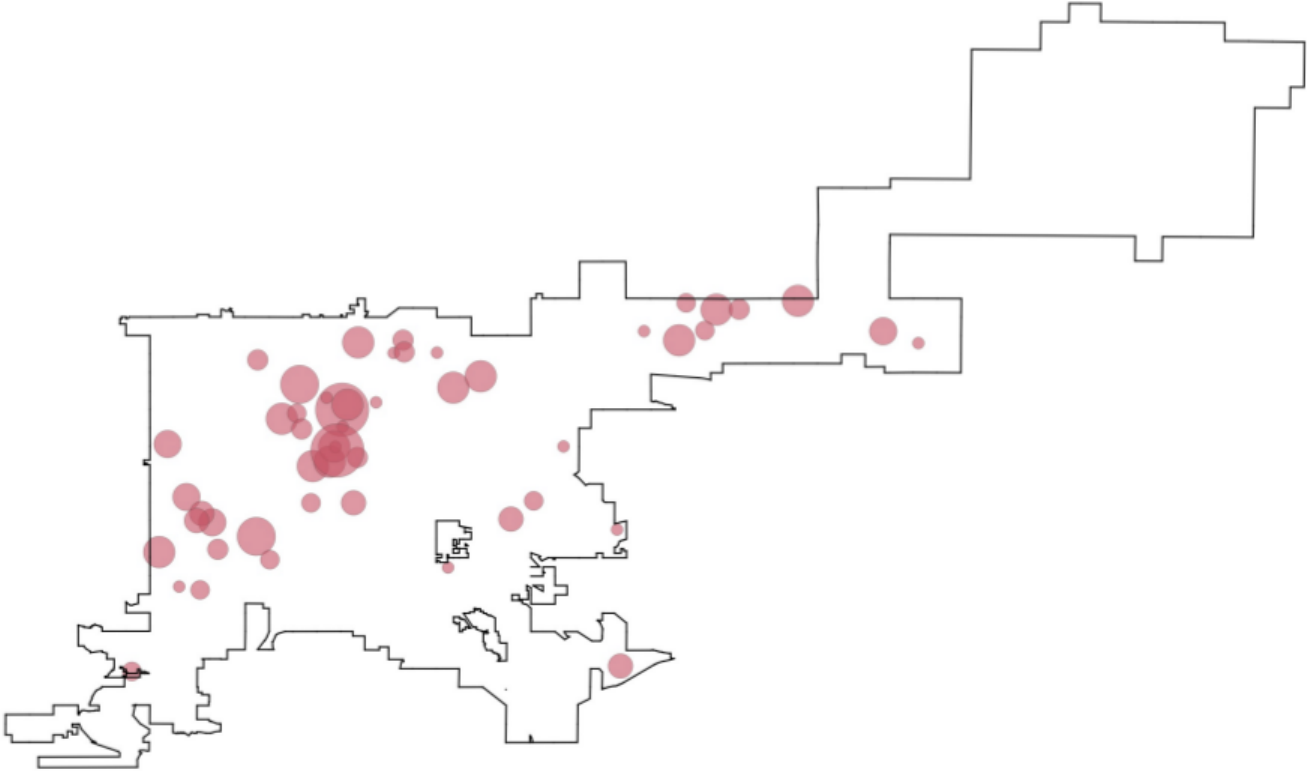}
\end{center}

\begin{center}
\includegraphics[scale = .4]{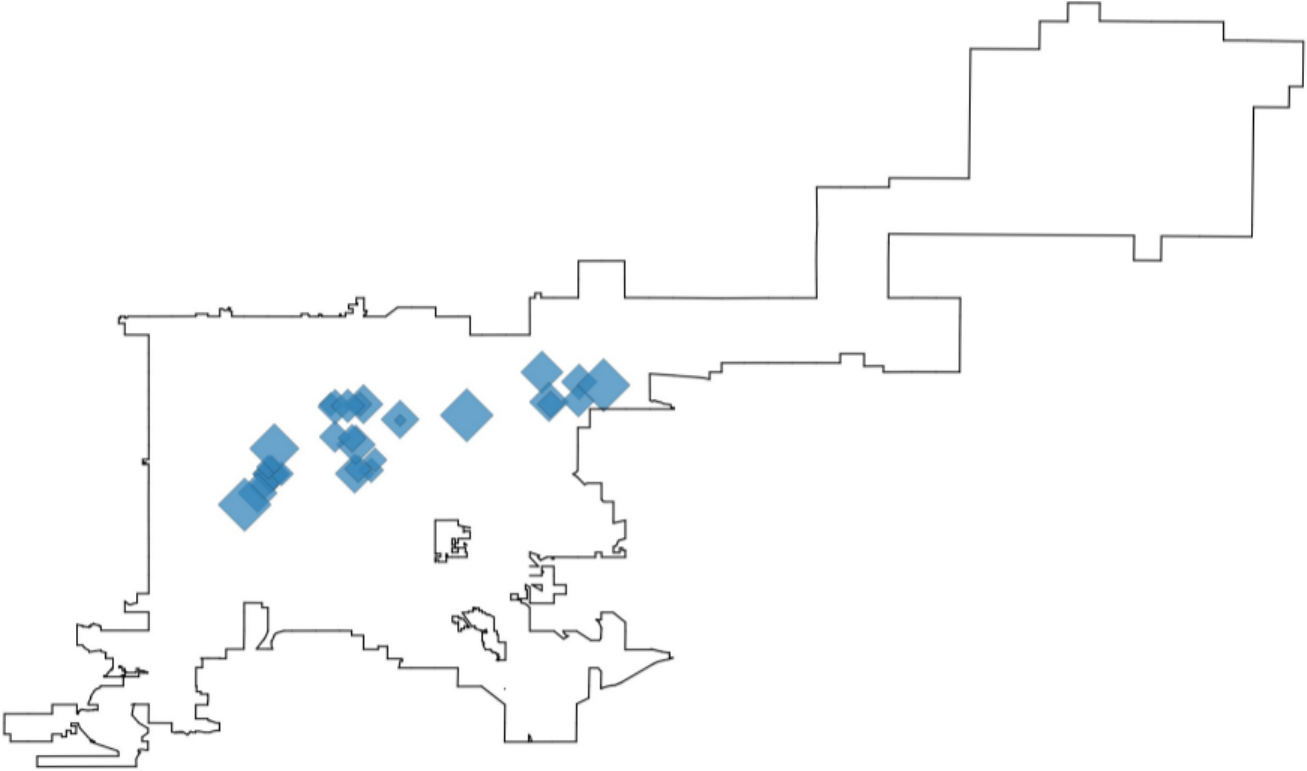}
\end{center}
\caption{Murder locations in Denver County in 2016 (top) and a barycenter indicating suggested police presence locations (bottom). }\label{fig:murders}
\end{figure}

\begin{table}
\begin{center}
\begin{tabular}{|c|c|c|c|c|c|c|}\hline
 & \multicolumn{3}{c|}{LP (\ref{baryLP}) } &  \multicolumn{3}{c|}{ LP (\ref{LPw}) }  \\ \hline
 Year & Constr & Vars & Time in Seconds & Constr & Vars & Time in Seconds \\ \hline
2012 & 1,520,675& 4,976,640 & 47,865.53 & 35 & 138,240 & 1.06  \\ \hline
2013 & 23,887,914 & 85,598,208 &* & 42 & 1,990,656 & 31.14 \\ \hline
2014 & 663,585 & 1,880,064 & 5,824.63 & 33 & 55,296 & 0.43 \\ \hline
2015 & 303,464,505 & 1,466,744,832&  * & 57 & 25,288,704 & 430.18\\ \hline
2016 & 26,127,412 & 115,395,840 & 223,886.9 & 52 & 2,177,280 & 19.15\\ \hline
\end{tabular}
\caption{Total running times comparing LP (\ref{baryLP}) and LP (\ref{LPw}) for data in general position. We see dramatic improvement using LP (\ref{LPw}), including computations that were not possible before (*).}\label{GenTotalRun}
\end{center}
\end{table}

This dramatic improvement is just as expected in view of the discussion in Section \ref{sec:datageneral}. LP (\ref{LPw}) is the smallest model for data in general position, and no sophisticated preprocessing is necessary to set it up. The other models are not competitive. The situation becomes more interesting for grid-based data, which we discuss next, and in more detail. 

\subsection{Grid-Structured Data: MNIST digits}\label{sec:digitcomp}

For computations of barycenters on the digits data set, we implemented LPs (\ref{baryLP}), (\ref{barymodLP}), (\ref{LPw}), and (\ref{LPhybrid}) using the set $S$, then also implemented LPs (\ref{baryLP}) and (\ref{barymodLP}) using the set $G$. Recall the discussion in Section \ref{sec:apriori}; the set $S$ requires the expensive processing of $S^*$, whereas $G$ does not. While using $G$ generates larger linear programs, it proves to be the much better approach when total times, including preprocessing, are considered. Representative sizes and times were formed by running computations for several sets of images for each digit and averaging the results.

\subsubsection{Computations on $S$: Slow Preprocessing, Small LPs\\}\label{sec:noprior}

We first examine computations in which we generate $S$ by processing each element of $S^*$ and check for duplication in those elements already produced. This requires exponential effort, but allows us to determine precisely which $\x_k $ can be generated by the support set of the digits under consideration. Overall, this slow preprocessing produces the smallest LPs.

Table \ref{SizeDigitS} shows the average number of constraints and variables for each formulation using this preprocessing strategy. 
The number of variables in LP (\ref{LPw}) is larger, by several orders of magnitude, than for the other models. Its only advantage is the lower number of constraints, but this effect is dominated by the dramatic increase in number of variables. Moving from LP (\ref{baryLP}) to LP (\ref{barymodLP}), there is about a $44\%$ reduction in variables.  As expected in light of the discussions in Section \ref{sec:modgrid} and \ref{sec:apriori}, LP (\ref{LPhybrid}) does show a consistent further improvement in model size, but the reduction in variables using LP (\ref{LPhybrid}) is relatively minor, typically less than $1\%$.

\begin{table}[t]
\begin{center}
\begin{tabular}{|c|c|c|c|c|c|c|c|c|c|}\hline
Digit & \multicolumn{2}{c|}{LP (\ref{baryLP})} &  \multicolumn{2}{c|}{LP (\ref{barymodLP})} &  \multicolumn{2}{c|}{LP (\ref{LPw})} &  \multicolumn{2}{c|}{LP (\ref{LPhybrid})} &$ |S|$ \\ \hline
& Constr & Vars& Constr & Vars& Constr & Vars& Constr & Vars & \\ \hline
0 & 10,598 & 1,317,437  & 10,598 & 740,116 &522 &281,656,116 & 10,130 & 739,615 & 2,519 \\ \hline
1 & 3,083 & 144,200 & 3,083 & 82,834 &199 & 6,058,800 & 2,867&82,624 & 721 \\ \hline
2 & 10,589 & 1,072,724 & 10,589 & 596,785 & 421 & 116,865,990& 10,001 & 596,191 & 2,542 \\ \hline
3 & 10,999 & 1,212,100 &10,999 & 679,939 & 459 & 167,401,080 & 10,503 & 679,442 & 2,635  \\ \hline 
4 & 8,964 & 801,204 & 8,964 & 443,410 & 372 & 74,412,360 & 8,260 & 442,725 & 2,148 \\ \hline
5 & 10,999 & 1,282,464 & 10,999 & 716,753 & 487 & 206,928,720 & 10,587 & 716,328 & 2,628  \\ \hline
6 &  7,443 & 617,352 & 7,443 & 343,541 & 347 & 53,316,900 & 7,047 & 343,135 & 1,774  \\ \hline
7 & 7,944 &640,974  & 7,944& 353,378& 336 &48,014,460 & 7,432 & 352,843 & 1,902 \\ \hline
8 & 10,104 & 1,193,794 & 10,104 & 661,078 & 496 & 211,403,136 & 11,964 & 1,424,899 & 2,867 \\ \hline
9 & 8,787 & 918,280 & 8,787 & 500,554 & 439 & 131,637,120& 8,207 & 499,940  & 2,087 \\ \hline 
\end{tabular}
\end{center}
\caption{The average number of constraints and variables for four measures of each digit, using set $S$. Using LP (\ref{baryLP}) as a baseline, LP (\ref{barymodLP}) is a significant improvement in size. LP (\ref{LPhybrid}) is slightly smaller than LP (\ref{barymodLP}). LP (\ref{LPw}) is not viable.}
\label{SizeDigitS}
\end{table}

Table \ref{RunDigitS} provides the {\em total} running times, with sub-categories of the {\em setup} (preprocessing) and {\em solution} times, for LP (\ref{baryLP}), LP (\ref{barymodLP}) and LP (\ref{LPhybrid}), as well as the LP solution times for LP (\ref{LPw}) when available. (Recall that that the internal pre-solving of Gurobi is turned off. The setup refers to our own routines to set up the model and load it to the solver.) The LP solution times are precisely as predicted based on the sizes of the programs: LP (\ref{LPw}) is several orders of magnitude slower than the other formulations; the (*) represent scenarios in which the LP solver was unable to complete due to memory constraints. LP (\ref{barymodLP}) and LP (\ref{LPhybrid}) both show improved solution times over LP (\ref{baryLP}). LP (\ref{LPhybrid}) consistently gives the fastest solution times, but the advantage over LP (\ref{barymodLP}) is not significant.

Comparing the total running times from Table \ref{RunDigitS}, as expected, the setup times dominate the solution times in contribution to the total. This setup cost increases from LP (\ref{baryLP}) to LP (\ref{barymodLP}), and then further to LP (\ref{LPhybrid}). LP (\ref{barymodLP}) is typically the best choice for overall time, due to a best tradeoff of decreased solution time to increased setup time. For LP (\ref{LPhybrid}), the minor reduction in size and solution times is insufficient to justify the increased setup time. 

\begin{table}[t]
\begin{center}
\begin{tabular}{|c|c|c|c|c|c|c|c|c|c|c|}\hline
\multicolumn{11}{|c|}{Time in Seconds} \\ \hline
Digit & \multicolumn{3}{c|}{LP (\ref{baryLP})} &  \multicolumn{3}{c|}{LP (\ref{barymodLP})} &  \multicolumn{3}{c|}{LP (\ref{LPhybrid})} & LP (\ref{LPw}) \\ \hline
& Setup & Solve & Total& Setup & Solve & Total& Setup & Solve & Total & Solve\\ \hline
0 & 604.25 & 141.44 & \textbf{746.23} & 688.97 & 102.94 & 792.09 & 745.62 & \textbf{92.95} & 838.77 & *   \\ \hline
1 & 4.57 & 2.49  & \textbf{7.09} & 5.95 & 1.39 & 7.36 & 7.09 & \textbf{1.20} & 8.32& 311.27 \\ \hline
2 & 304.10 & 125.26 & 429.68 & 337.77 & 58.44 & 396.37 & 349.81 & \textbf{41.95} & \textbf{391.91}& 18,464.7 \\ \hline
3 & 402.63 & 152.85 & 555.83  & 448.20 & 71.81 & \textbf{520.18} & 480.72 & \textbf{68.29} & 549.23 & *\\ \hline
4 & 141.49 & 83.20 & 224.96 & 159.98 & 33.88 & 193.97 & 184.60 & \textbf{33.52} & \textbf{182.23}& 12,224.8 \\ \hline
5 & 467.24 & 163.34 & 630.99 & 535.25 & 86.74 & \textbf{622.24} & 582.32 & \textbf{61.11} & 643.59& * \\ \hline
6 & 92.21 & 45.09 & 137.50 & 104.93 & 24.79 & \textbf{129.80}  & 116.76 & \textbf{17.48} & 134.35 & 2,337.17  \\ \hline
7 & 69.82 & 53.71 & 123.71 & 81.64 & \textbf{21.41} & \textbf{103.13} & 91.18 & 21.82 & 113.07 &  6,060.96 \\ \hline
8 & 375.41 & 114.70 & \textbf{490.47} & 438.58 & 91.61 & 530.36 & 480.66 & \textbf{51.69} & 532.52 &* \\ \hline
9 &226.72  & 92.44   & 319.45 & 263.57 & 41.83 & \textbf{305.52} & 289.74 & \textbf{34.56} & 324.42 &*\\ \hline
\end{tabular}
\end{center}\caption{Average times for four measures of each digit, using set $S$. For all digits, the size reduction in LPs (\ref{barymodLP}) and (\ref{LPhybrid}) offers significant savings in solution times. Fastest solution and total times are displayed in bold. LP (\ref{LPw}) is not competitive; solution times are several orders of magnitude larger. Due to the processing of $S^*$, the times for LP (\ref{baryLP}), LP (\ref{barymodLP}) and LP (\ref{LPw}) are dominated by the setup time. For most digits, LP (\ref{barymodLP}) offers the best total time.} \label{RunDigitS}
\end{table}

\subsubsection{Computations on $G$: Fast Preprocessing, Slightly Larger LPs\\}\label{sec:prior}

Next, we show results for the same collections of digits for LP (\ref{baryLP}) and LP (\ref{barymodLP}), using the easy-to-preprocess subgrid $G$ of $G_{\text{full}}$ instead of $S$. Recall the discussion in Section \ref{sec:apriori} regarding the underlying structure which reveals that an implementation of LP (\ref{LPhybrid}) for $G$ is not promising. 

Since $S \subseteq G$, the linear programs using $G$ are larger; however, as $G$ is significantly easier to generate, the preprocessing effort is greatly reduced. In Table \ref{SizeDigitG}, we provide $|G|$ and $|S|$ for our representative digits, along with the corresponding program sizes. There is about a $35\%$ increase in variables for LP (\ref{barymodLP}) on $G$ compared to the corresponding formulation on $S$. We also note that the number of variables for LP (\ref{barymodLP}) on $G$ is smaller than for LP (\ref{baryLP}) on $S$ for all our sample runs.


\begin{table}[t]
\begin{center}
\begin{tabular}{|c|c|c|c|c|c|c|c|c|c|c|}\hline
& \multicolumn{5}{c|}{$G$} & \multicolumn{5}{c|}{$S$} \\ \hline
& \multicolumn{2}{c|}{LP (\ref{baryLP})}  &  \multicolumn{2}{c|}{LP (\ref{barymodLP})} & $|G|$ & \multicolumn{2}{c|}{LP (\ref{baryLP})} &  \multicolumn{2}{c|}{LP (\ref{barymodLP})} & $|S|$ \\ \hline
Digit & Constr & Vars& Constr & Vars & &Constr & Vars& Constr & Vars &\\ \hline
0 & 12,478 & 1,563,247 & 12,478 & 1,139,427 & 2,989 & 10,598 & 1,317,437  & 10,598 & 740,116 & 2,519 \\ \hline
1 & 3,371 & 158,600 & 3,371 & 117,684 & 793 & 3,083 & 144,200 & 3,083 & 82,834 & 721\\ \hline
2 & 13,417  & 1,371,078 & 13,417 & 999,274 & 3,249 & 10,589 & 1,072,724 & 10,589 & 596,785 & 2,542\\ \hline
3 & 12,903 & 1,431,060 & 12,903 & 1,026,691 & 3,111 & 10,999 & 1,212,100 & 10,999 & 679,939 & 2,635\\ \hline
4 & 11,840 & 1,069,391 & 11,840 & 776,727 & 2,867 & 8,964 & 801,204 & 8,964 & 443,410 & 2,148  \\ \hline
5 & 11,955 & 1,399,096 & 11,955 & 1,016,077 & 2,867 & 10,999 & 1,282,464 & 10,99 & 715,753 & 2,628 \\ \hline
6 & 8,887 & 742,980 & 8,887 & 541,330 & 2,135 & 7,443 & 617,352 & 7,443 & 343,541 & 1,774 \\ \hline
7 & 10,340 & 842,837  & 10,340  & 605,340 & 2,501 & 7,944 & 640,974 & 7,944 & 353,378 & 1,902 \\ \hline
8 & 11,964 & 1,424,899 & 11,964 & 1,029,485 & 2,867 & 10,104 & 1,193,794 & 10,104 & 661,078 & 2,402 \\ \hline
9 & 10,931 & 1,154,120 & 10,931 & 818,342 & 2,623 & 8,787 & 918,280 & 8,787 &500,554 & 2,087  \\ \hline
\end{tabular}
\end{center}\caption{The average number of constraints and variables for four measures of each digit using the easily generated, but larger, set $G$ versus using the difficult to generate, but smaller set $S$. The problem sizes on $G$ are approximately $35\%$ larger.}\label{SizeDigitG}
\end{table}

The advantage of using set $G$ is immediately apparent in Table \ref{SolveDigitG}. The setup of both LPs takes only a few seconds and is almost negligible in view of the total running time. Unlike for set $S$, the setup time for LP (\ref{barymodLP}) on $G$ decreases slightly from the already low setup time for LP (\ref{baryLP}) on $G$:  The fact that there are fewer $y$-variables in LP (\ref{barymodLP}) now is a direct advantage and becomes visible in the setup time. As expected, the solution times for LP (\ref{barymodLP}) on $G$ also are improvements over the solution times for LP (\ref{baryLP}).

Finally, in Table \ref{RunDigitG} we see the significant and consistent improvement in total running times when moving from LPs on $S$ to LPs on $G$. The slower LP solution times on $G$ are greatly outweighed by the reduction in setup times. We see up to a $75\%$ reduction in total running time from LP (\ref{baryLP}) on $S$ to LP (\ref{barymodLP}) on $G$.

\begin{table}[t]
\begin{center}
\begin{tabular}{|c|c|c|c|c|c|c|}\hline
& \multicolumn{6}{c|}{Time in Seconds} \\ \hline
Digit &\multicolumn{3}{c|}{ LP (\ref{baryLP}) }&\multicolumn{3}{c|}{ LP (\ref{barymodLP})} \\ \hline
& Setup & Solve & Total & Setup & Solve & Total \\ \hline
0 & 9.38  & 233.97& 243.64 & \textbf{7.92} & \textbf{192.63} &\textbf{200.89}\\ \hline
1 & 0.42  & 2.94& 3.18 & \textbf{0.39} & \textbf{2.58} & \textbf{3.00}\\ \hline
2 & 5.00  & 256.18 & 261.35 & \textbf{4.08}& \textbf{208.43} &\textbf{212.79} \\ \hline
3 & 6.73   & 282.34 &289.25 & \textbf{5.28}& \textbf{141.66} &\textbf{147.25} \\ \hline 
4 & 3.77  & 146.12 & 149.99 &\textbf{2.97}& \textbf{119.97} & \textbf{123.19} \\ \hline
5 & 7.31 & 201.88 & 209.39 & \textbf{6.10} & \textbf{146.32} &\textbf{152.73}  \\ \hline
6 & 2.65  & 55.87 & 58.51 & \textbf{2.10} & \textbf{45.32} &\textbf{47.58} \\ \hline
7 & 2.71 & 107.41& 110.13 &\textbf{2.03}  & \textbf{52.14} &\textbf{54.35}\\ \hline
8 & 7.53  &  197.01 &204.75 & \textbf{6.41}& \textbf{184.54} &\textbf{191.28} \\ \hline
9 & 5.06 & 172.58& 177.79 &\textbf{4.47}  & \textbf{127.23} &\textbf{131.96} \\ \hline
\end{tabular}
\end{center}
\caption{Average setup, solution, and total running times for four measures of each digit using the easily generated, but larger, set $G$. Due to the simplicity of constructing the set $G$, the setup times are consistently low and contribute far less to the total running times than the solution times. LP (\ref{barymodLP}) outperforms LP (\ref{baryLP}) in both setup and solution times. Fastest times are displayed in bold.}\label{SolveDigitG}
\end{table}

\begin{table}[t]
\begin{center}
\begin{tabular}{|c|c|c|c|c|c|c|}\hline
 \multicolumn{5}{|c|}{Total Running Time in Seconds} \\ \hline
Digit &\multicolumn{2}{c|}{LP (\ref{baryLP})}  & \multicolumn{2}{c|}{LP (\ref{barymodLP})}\\ \hline
&$S$ & $G$ & $S$ & $G$  \\ \hline
0 & 746.23 & 243.62 & 792.09& \textbf{200.89}  \\ \hline
1 &7.09 &  2.94 &7.36 & \textbf{2.58}   \\ \hline
2 & 429.68 & 261.35 & 396.37& \textbf{212.79}  \\ \hline
3 & 555.83& 289.25 & 520.18  &  \textbf{147.25}\\ \hline
4 & 224.97& 149.99 & 193.97 & \textbf{123.19}  \\ \hline
5  & 630.99& 209.39 & 622.24  & \textbf{152.73} \\ \hline
6 & 137.50& 58.51 & 129.80  & \textbf{47.58} \\ \hline
7 & 123.71 & 110.13& 103.13 & \textbf{54.35}   \\ \hline
8 & 490.47 & 204.75& 530.36  & \textbf{191.28}   \\ \hline
9 & 319.45 & 177.79& 305.52 & \textbf{131.96}   \\ \hline
\end{tabular}
\end{center}\caption{Average total running times including setup for four measures of each digit, using the easily generated set $G$ versus  the difficult to generate set $S$. LP (\ref{barymodLP}) on $G$ performs best by a significant margin in all runs. Fastest total times are displayed in bold. } \label{RunDigitG}
\end{table}

\subsection{Hybrid Data: An MNIST-style example}\label{sec:besthybrid}

For our final computational experiments, we consider an example in which LP (\ref{LPhybrid}) performs significantly better than the other three formulations. For these experiments, we compute a barycenter (using the exact $S$) for four measures representing the letter ``i''. The main difference of these (handwritten) letters lies in the location of the dot, rather than the line at the base of the letter. Thus, we desire more accuracy for the upper portion of the image. To this end, we record the letters in a combination of two grids: the base is recorded in a $16 \times 16$ grid, in the style of the MNIST digits. The dot is recorded in what we call the {\em upper grid}. We compare the LP formulations as we refine this upper grid, starting from the original $16 \times 16$ grid; refinement 2 instead introduces support points from a $32 \times 32$ grid with half the step size. Two sample letters are shown in Figure \ref{fig:Samplei}.

\begin{figure}[t]
\begin{center}
\includegraphics[scale = .6]{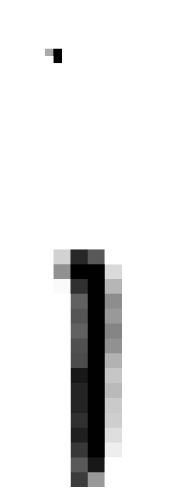}
\hspace{.4in}
\includegraphics[scale = .6]{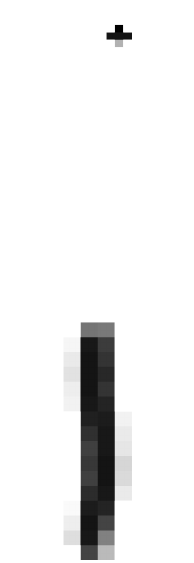}
\end{center}
\caption{Sample letters ``i'' with dots mapped to a two-times-finer upper grid.}\label{fig:Samplei}
\end{figure} 

Throughout the experiments, we hold the number of support points in the $P_i$ constant, i.e., the dot is always represented through the same number of support points. This allows a comparison that isolates the effects of the refinement of the upper grid itself on the resulting formulation sizes. However, the sizes of the constructed LPs vary significantly, as seen in Table \ref{SizeDot}, due to the changing amount of repetition of weighted means when producing $S$. For LP (\ref{baryLP}) and LP (\ref{barymodLP}), the largest programs arise for a three-times-finer grid -- this is not surprising, as refinements by $2$ or $4$ naturally lead to more repetition. LP (\ref{LPw}) does not account for this repetition, so its size remains constant.

As expected,  LP (\ref{LPhybrid}) always provides an improvement over LP (\ref{baryLP}) and LP (\ref{barymodLP}) in model size. Essentially, it treats the few support points of the dot like data in general position through the introduction of $w$-variables. As we refine the grid, the size reduction for LP (\ref{LPhybrid}) over LP (\ref{barymodLP}) becomes more and more significant. Note that much of this reduction actually lies in the number of constraints.

\begin{table}[t]
\begin{center}
\begin{tabular}{|c|c|c|c|c|c|c|c|c|}\hline
Grid Refinement & \multicolumn{2}{c|}{LP (\ref{baryLP})}& \multicolumn{2}{c|}{LP (\ref{barymodLP})}& \multicolumn{2}{c|}{LP (\ref{LPw}) }& \multicolumn{2}{c|}{LP (\ref{LPhybrid})} \\ \hline
& Constr & Vars& Constr & Vars& Constr & Vars& Constr & Vars \\ \hline
1 & 7,067 & 370,008 & 7,067 & 194,632 &215 & 8,258,112 &5,823 & 192,901 \\ \hline
2  & 20,243 & 1,081,512 & 20,243 & 476,316 & 215 & 8,258,112 & 14,883 & 469,330 \\ \hline
3& 55,971 & 3,010,824& 55,971 & 795,175& 215 & 8,258,112&17,427 & 665,781  \\ \hline
4 &  48,623 & 2,614,032 & 48,623 & 794,117 & 215 & 8,258,112& 17,027 & 686,133\\ \hline
\end{tabular}
\end{center}
\caption{LP sizes for each formulation as the upper grid is refined.  LP (\ref{LPw}) is always the same, very large size regardless of refinement. As we refine the upper grid, the reduction in size using LP (\ref{LPhybrid}) over LP (\ref{barymodLP}) becomes more significant.}\label{SizeDot}
\end{table}

Setup, solution, and total running times are shown in Table \ref{SolveDot}.  In this scenario, the setup times are dominated by the solution times. As the upper grid is refined, LP (\ref{baryLP}) scales quite poorly in solution times, and thus also in total time. LP (\ref{barymodLP}) performs better, but is still vastly outperformed by LP (\ref{LPhybrid}): The combination of slightly fewer variables and signficiantly fewer constraints leads to much better solution times, especially for the finest upper grids. 

Since LP (\ref{LPw}) does not account for repetition, its size is constant through the experiments and its solution times remain constant. This essentially provides an upper bound on the running time for this number of measures regardless of refinement level; LP (\ref{LPhybrid}) always performs better, as it can take advantage of the still-present repetition, in particular in the base of the letters. Compared to LP (\ref{baryLP}), we ultimately see as much as a $99\%$ improvement in total running time using LP (\ref{LPhybrid}). Further, the total running time for LP (\ref{LPhybrid}) is approximately a $66\%$ improvement over LP (\ref{barymodLP}) and LP (\ref{LPw}). 

\begin{table}[t]
\begin{center}
\resizebox{\columnwidth}{!}{
\begin{tabular}{|c|c|c|c|c|c|c|c|c|c|c|c|c|}\hline
 \multicolumn{13}{|c|}{Time in Seconds} \\ \hline
{Dot Grid} & \multicolumn{3}{c|}{LP (\ref{baryLP})} & \multicolumn{3}{c|}{LP (\ref{barymodLP})}  & \multicolumn{3}{c|}{LP (\ref{LPw})} & \multicolumn{3}{c|}{LP (\ref{LPhybrid})} \\ \hline
  &Setup & Solve & Total &Setup & Solve & Total &Setup & Solve & Total&Setup  & Solve & Total \\ \hline 
1 &9.06 & 29.11 & 38.19 & 10.65 & 6.90 &   \textbf{17.59} & 19.28 & 419.87&  443.85 &  13.15 & \textbf{6.35}&  19.55   \\ \hline
2 & 18.50 &1056.37 &1075.22& 20.82 &  69.24& 90.18  & 19.31 &415.63&438.98 & 26.81 & \textbf{57.87} &\textbf{84.59}\\ \hline
3& 33.74 & 14111.3 & 14146& 31.63 & 461.32& 493.15 & 19.30 & 424.76 & 448.82  & 46.59 & \textbf{92.65}&\textbf{139.32} \\ \hline
4 &34.45 & 10643.3 & 10687.6& 33.94 & 462.90&495.22  & 19.07 &437.11& 461.10  & 46.24 &\textbf{118.35} & \textbf{164.78}\\ \hline
\end{tabular}
}
\end{center}
\caption{Setup, solution, and total running times for each formulation as the upper grid is refined. LP (\ref{LPhybrid}) significantly outperforms the others, scaling far better with the refinement. Fastest solution and total times are displayed in bold.}\label{SolveDot}
\end{table}

\section{Concluding Remarks} \label{sec:conc}

In this paper, we devised new, smaller linear programs for exact solutions to the Discrete Barycenter Problem. Each formulation is an improvement over the previously known linear program. Each is the best choice for different types of data, as seen in theoretical analysis and computational experiments.

With an abundance of applications for the barycenter problem, these better, exact models may play an important role in many current research questions. In addition to the inherent benefit of finding an exact solution for larger problem instances, they also can be used to improve LP-based approximation methods, and are highly beneficial for evaluating the quality of state-of-the-art heuristic approaches. 

There are a couple of natural questions arising from this work. Our computational experiments are designed to highlight the relative performance of each linear program. For practical computations, we still see significant potential for improvement in a somewhat different type of study. For example, we found that Gurobi's barrier method makes particularly good use of multiple cores when solving these problems. A more technical analysis of the behavior of the programs with respect to different solvers may be of interest to optimize total running times. The similarity to classical transportation problems also suggests that the problem might be approachable through parallelized algorithms. Further, the large number of variables and relatively low number of constraints in the models, especially for LP (\ref{LPw}), makes a column generation approach promising. Finally, a formal proof of hardness of the Discrete Barycenter Problem is notably missing in the literature -- in this paper, we only proved that the setup of the LP formulations is hard. 
\section*{Acknowledgments}

We thank Natalia Villegas Franco for the implementation of a visualization basis for the computations on Denver crime data depicted in Figure \ref{fig:murders}. We gratefully acknowledge support through the Collaboration Grant for Mathematicians  {\em Polyhedral Theory in Data Analytics} of the Simons Foundation.

\section*{Biographies}

Steffen Borgwardt works in the Department of Mathematical and Statistical Sciences at the University of Colorado Denver. His research lies on the intersection of combinatorial optimization, polyhedral theory, and linear programming. He is a lifetime Humboldt fellow and received the EURO Excellence in Practice Award 2013 for his work on optimization in land consolidation (jointly with Dr. Andreas Brieden, Dr. Peter Gritzmann).

Stephan Patterson is a Ph.D. student at the University of Colorado Denver. He has a background in computational mathematics and optimization, and holds an M.S. in Applied Mathematics, with an emphasis in Numerical Analysis. 

The ideas presented in this paper emerged during ongoing work on a divide-and-conquer approach that would provide an approximate solution to the discrete barycenter problem. Surprising results of computational experiments for grid-structured data highlighted a path to the improvements on the state-of-the-art for the exact computation of discrete barycenters found in this paper.

\bibliography{barycenters_literature}
\bibliographystyle{plain}

\end{document}